\documentclass[12pt,a4paper,twoside,reqno]{amsart}
\usepackage[a4paper,top=3cm,bottom=3cm,inner=2.5cm,outer=2.5cm]{geometry}

\usepackage[utf8]{inputenc}
\usepackage{amsmath,amssymb,amsthm,amsfonts}
\usepackage[english]{babel}
\usepackage{multirow,enumitem,array}
\usepackage{xparse}
\usepackage{pst-node,placeins,xspace,fix-cm}
\usepackage{tikz-cd}
\usepackage{accents}
\usepackage[nomessages]{fp}
\usepackage[colorlinks,linkcolor=blue,citecolor=violet,urlcolor=darkgray,final,hyperindex,linktoc=page,hyperfootnotes=true]{hyperref}
\newcolumntype{F}{>{$}c<{\hspace{-0.9ex}$}}
\newcolumntype{:}{>{$}m{0.8ex}<{$}}
\newcolumntype{R}{>{$}r<{$}}
\newcolumntype{C}{>{$}c<{$}}
\newcolumntype{L}{>{$}l<{$}}
\newcolumntype{N}{@{}>{$}l<{$}}
\newcommand{\linesep}[1]{\renewcommand{\arraystretch}{#1}}
\newlength\horspace
\newcommand{\h}[1][1.0]{\hspace{#1\horspace}}
\newlength\verspace
\renewcommand{\v}[1][1.0]{\vspace{#1\verspace}\xspace}
\newlength\negverspace
\newcommand{\nv}[1][1.0]{\vspace{#1\negverspace}\xspace}
\tikzset{iso/.style={draw=none,every to/.append style={edge node={node [sloped, allow upside down, auto=false]{$\cong$}}}}}
\tikzset{adjunction/.style={draw=none,every to/.append style={edge node={node [sloped, allow upside down, auto=false]{$\dashv$}}}}}
\tikzset{simeq/.style={draw=none,every to/.append style={edge node={node [sloped, allow upside down, auto=false]{$\simeq$}}}}}
\tikzset{simeqS/.style={draw=none,every to/.append style={edge node={node [sloped, allow upside down, auto=false]{$\raisebox{0.8em}{$\simeq$}$}}}}}
\tikzset{aiso/.style={simeqS,preaction={draw,->}}}
\tikzset{dotdot/.style={dash pattern=on 0.25ex off 0.2ex, dash phase=0ex}}
\tikzset{RightA/.style={double distance=3.5pt,>={Implies},->},%
	triple/.style={-,preaction={draw,RightA}},%
	quadruple/.style={preaction={draw,RightA,shorten >=0pt},shorten >=1pt,-,double,double distance=0.2pt}}

\newtheorem{teor}{Theorem}[section]
\newtheorem{coroll}[teor]{Corollary}

\newtheorem{prop}[teor]{Proposition}

\theoremstyle{definition}
\newtheorem{defne}[teor]{Definition}

\newtheorem{rem}[teor]{Remark}
\newtheorem{exampl}[teor]{Example}
\newtheorem{rec}[teor]{Recall}
\newtheorem{cons}[teor]{Construction}

\def\nameit#1{\textrm{#1}~}
\def\defx{\nameit{Definition}}
\def\thex{\nameit{Theorem}}
\def\prox{\nameit{Proposition}}
\def\conx{\nameit{Construction}}
\def\remx{\nameit{Remark}}
\def\recx{\nameit{Recall}}
\def\corx{\nameit{Corollary}}

\def\exax{\nameit{Example}}

\NewDocumentEnvironment{cd}{s O{7} O{7} b}{%
	\IfBooleanF{#1}{\begin{equation*}}\begin{tikzcd}[row sep=#2ex,column sep=#3ex,ampersand replacement=\&]
			#4
		\end{tikzcd}\IfBooleanF{#1}{\end{equation*}}\ignorespacesafterend}{}

\newenvironment{enum}{\begin{enumerate}[label=$($\hspace{0.12ex}\roman*\hspace{0.075ex}$)$]}{\end{enumerate}}
\newenvironment{enumT}{\begin{enumerate}[label=$($\hspace{-0.1ex}\roman*\hspace{0.13ex}$)$]}{\end{enumerate}}
\newenvironment{fun}{\[\begin{tabular}{F:RCL}}{\end{tabular}\]\ignorespacesafterend}
\newenvironment{eqD}[1]{\begin{equation}\label{#1}}{\end{equation}\ignorespacesafterend}
\newenvironment{eqD*}{\begin{equation*}}{\end{equation*}\ignorespacesafterend}

\def\:{\colon}
\providecommand\ordinarycolon{:}
\def\vcentcolon{\mathrel{\mathop\ordinarycolon}}
\newcommand{\deq}{\mathrel{\vcentcolon\mkern-1.2mu=}}
\def\phi{\varphi}
\def\eps{\varepsilon}

\newcommand{\romanuppercase}[1]{\uppercase\expandafter{\romannumeral #1\relax}}

\newcommand{\p}[1]{\big(\mkern1mu{#1}\mkern1mu\big)}
\newcommand{\pteor}[1]{$($\h[-1.9]{#1}\h[0.6]$)$}

\def\dfn#1{{\bfseries\itshape #1}}
\def\predfn#1{{\itshape #1}}
\newcommand{\refs}[1]{$($\ref{#1}$)$}

\newcommand{\scaleu}[2][1.2]{{\scalebox{#1}{$#2$}}}

\newcommand{\ov}[1]{\overline{#1}}
\def\t#1{\widetilde{#1}}

\newcommand{\tight}{_\tau}
\newcommand{\loose}{_\lambda}

\def\c{\circ}

\newcommand{\st}{^{\ast}}
\newcommand{\stb}{_{\ast}}

\newcommand{\sliceslant}[2]{{\raisebox{.1em}{$#1$}\mkern-1mu\left/\raisebox{-.25em}{$#2$}\right.}}

\newcommand{\laxsliceslant}[2]{{\raisebox{.1em}{$#1$}\mkern-1mu\left/_{\mkern-4.1mu\operatorname{lax}}\right.\raisebox{-.25em}{$#2$}}}

\DeclareFontFamily{OT1}{pzc}{}
\DeclareFontShape{OT1}{pzc}{m}{it}{<->s*[1.21]pzcmi7t}{}
\DeclareMathAlphabet{\mathpzc}{OT1}{pzc}{m}{it}

\DeclareFontFamily{U}{dutchcal}{\skewchar\font=45}
\DeclareFontShape{U}{dutchcal}{m}{n}{<->s*[1.05] dutchcal-r}{}
\DeclareMathAlphabet{\mathlcal}{U}{dutchcal}{m}{n}
\newcommand{\catfont}[1]{\ensuremath{\mathpzc{#1}}\xspace}

\newcommand{\A}{\catfont{A}}
\newcommand{\B}{\catfont{B}}
\newcommand{\C}{\catfont{C}}
\newcommand{\D}{\catfont{D}}
\newcommand{\E}{\catfont{E}}
\newcommand{\F}{\catfont{F}}
\newcommand{\K}{\catfont{K}}

\newcommand{\J}{\catfont{J}}
\def\P{\catfont{P}}

\newcommand{\T}{\catfont{T}}

\newcommand{\X}{\catfont{X}}

\def\L{\catfont{L}}
\def\S{\catfont{S}}

\newcommand{\1}{\catfont{1}}
\newcommand{\2}{\catfont{2}}

\newcommand{\Set}{\catfont{Set}}

\newcommand{\Cat}{\catfont{Cat}}
\newcommand{\CAT}{\catfont{Cat}}

\NewDocumentCommand{\Sh}{o m}{
	\ensuremath{\catfont{Sh}\hspace{-0.15ex}\left({#2}\IfNoValueF{#1}{,{#1}}\right)}
}

\NewDocumentCommand{\Alg}{t+ t' m}{
	\ensuremath{\IfBooleanT{#1}{\catfont{Ps}\mbox{-}}{#3}\mbox{-}\catfont{\IfBooleanT{#2}{Co}Alg}}
}

\newcommand{\slice}[2]{\sliceslant{#1}{#2}}

\newcommand{\laxslice}[2]{\laxsliceslant{#1}{#2}}

\newcommand{\HomC}[3]{{#1}\left({#2},\h[1]{#3}\right)}
\newcommand{\m}[2]{\ensuremath{\left[#1,#2\right]}\xspace}

\newcommand{\moplaxn}[2]{\ensuremath{\left[#1,#2\right]_{\oplaxn}}\xspace}
\newcommand{\moplaxm}[2]{\ensuremath{\left[#1,#2\right]_{\oplaxm}}\xspace}
\newcommand{\mF}[2]{\ensuremath{\left[#1,#2\right]^{\F}}\xspace}
\newcommand{\mFoplax}[2]{\ensuremath{\left[#1,#2\right]^{\F}_{\oplax}}\xspace}
\newcommand{\mFlax}[2]{\ensuremath{\left[#1,#2\right]^{\F}_{\lax}}\xspace}
\newcommand{\mFopolax}[2]{\ensuremath{\left[#1,#2\right]^{\F}_{\opn{(op)lax}}}\xspace}

\newcommand{\opn}[1]{\operatorname{#1}}
\newcommand{\y}[1]{\ensuremath{\operatorname{y}\hspace{-0.2ex}\left({#1}\right)}}

\DeclareMathOperator{\yy}{y}

\newcommand{\id}[1]{\operatorname{id}_{#1}}
\newcommand{\Id}[1]{\operatorname{Id}_{#1}}

\newcommand{\op}{\ensuremath{^{\operatorname{op}}}}

\newcommand{\x}[1][]{\h[-1]\times_{#1}\h[-1]}
\newcommand{\restr}[2]{{\left.\kern-\nulldelimiterspace {#1}\vphantom{\big|} \right|_{#2}}}

\newcommand{\dom}{\operatorname{dom}}

\newcommand{\pr}[1]{\operatorname{pr}_{#1}}

\newcommand{\cart}[3][]{{\operatorname{Cart}_{#1}\hspace{-0.2ex}\left({#2},\h{#3}\right)}}

\DeclareMathOperator{\colim}{colim}
\DeclareMathOperator{\bbicolim}{bicolim}
\DeclareMathOperator{\pseudo}{pseudo}
\DeclareMathOperator{\lax}{lax}
\DeclareMathOperator{\oplax}{oplax}

\DeclareMathOperator{\oplaxn}{oplax^{cart}}
\DeclareMathOperator{\oplaxm}{oplax^{mark}}

\DeclareMathOperator{\ssigma}{sigma}

\newcommand{\wcolim}[2]{{\colim}^{#1}\h{#2}}
\newcommand{\bicolim}[2]{{\bbicolim}^{#1}\h{#2}}
\newcommand{\pseudocolim}[2]{\pseudo\mbox{-}\h[1.5]\wcolim{#1}{#2}}

\newcommand{\oplaxncolim}[2][\Delta 1]{\oplaxn\mbox{-}\h[1.5]\wcolim{#1}{#2}}

\newcommand{\oplaxFcolim}[2]{\oplax^\F\mbox{-}\h[1.5]\wcolim{#1}{#2}}
\newcommand{\sigmacolim}[1]{\ssigma\mbox{-}\h[1.5]\wcolim{\Delta 1}{#1}}
\newcommand{\sigmabicolim}[1]{\ssigma\mbox{-}\h[1.5]\bicolim{\Delta 1}{#1}}

\newcommand{\Int}[1]{\ensuremath{\int \hspace{-0.35ex} #1}}
\newcommand{\Intdiag}[1]{\ensuremath{\scaleu{\int} \hspace{-0.15ex} #1}}

\newcommand{\Groth}[1]{\Int{#1}}
\newcommand{\Grothdiag}[1]{\Intdiag{#1}}

\newcommand{\groth}[1]{\mathcal{G}\mkern-1.4mu\left(#1\right)}
\def\H{\ensuremath{\mathcal{H}}}

\newcommand{\too}{\longrightarrow}
\newcommand{\mto}{\mapsto}

\makeatletter
\newcommand{\ar}[2][]{\xrightarrow[#1]{#2}}
\newcommand{\aar}[2][]{\xrightarrow[#1]{#2}} % in tikz
\def\xlongrightarrowfill@{\arrowfill@\relbar\relbar\longrightarrow}
\newcommand{\arr}[2][]{%
	\ext@arrow 0099\xlongrightarrowfill@{#1}{#2}}
\newcommand{\aarr}[2][]{%
	\ext@arrow 0099\xlongrightarrowfill@{#1}{#2}} % in tikz

\newcommand{\aR}[2][]{%
	\ext@arrow 0055{\Rightarrowfill@}{#1}{#2}}
\def\xLongrightarrowfill@{\arrowfill@\Relbar\Relbar\Longrightarrow}
\newcommand{\aRR}[2][]{%
	\ext@arrow 0099\xLongrightarrowfill@{#1}{#2}}
\def\aitofill@{\arrowfill@{\lhook\joinrel\relbar}\relbar\rightarrow}
\newcommand{\aito}[2][]{%
	\ext@arrow 3095\aitofill@{#1}{#2}}
\def\Longaitofill@{\arrowfill@{\lhook\joinrel\relbar\joinrel\relbar}\relbar\rightarrow}
\newcommand{\aitoo}[2][]{%
	\ext@arrow 0099\Longaitofill@{#1}{#2}}

\newcommand{\al}[2][]{\xleftarrow[#1]{#2}}
\def\xlongleftarrowfill@{\arrowfill@\longleftarrow\relbar\relbar}
\newcommand{\all}[2][]{%
	\ext@arrow 0099\xlongleftarrowfill@{#1}{#2}}
\newcommand{\aL}[2][]{%
	\ext@arrow 0055{\Leftarrowfill@}{#1}{#2}}
\def\xLongleftarrowfill@{\arrowfill@\Longleftarrow\Relbar\Relbar}
\newcommand{\aLL}[2][]{%
	\ext@arrow 0099\xLongleftarrowfill@{#1}{#2}}
\def\xmapstofill@{\arrowfill@{\mapstochar\relbar}\relbar\rightarrow}
\newcommand{\am}[2][]{%
	\ext@arrow 0395\xmapstofill@{#1}{#2}}
\def\xlongmapstofill@{\arrowfill@\relbar\relbar\longmapsto}
\newcommand{\amm}[2][]{%
	\ext@arrow 0399\xlongmapstofill@{#1}{#2}}

\newcommand{\eqq}{\DOTSB\protect\Relbar\protect\joinrel\Relbar}
\def\xeqqfill@{\arrowfill@\Relbar\Relbar\eqq}
\newcommand{\aeqq}[2][]{%
	\ext@arrow 0099\xeqqfill@{#1}{#2}}
\def\xRrightarrowfill@{\arrowfill@\equiv\equiv\Rrightarrow}
\newcommand{\aM}[2][]{\ext@arrow 0359\xRrightarrowfill@{#1}{#2}}
\newcommand{\Llongrightarrow}{%
	\DOTSB\protect\equiv\protect\joinrel\Rrightarrow}
\def\xLlongrightarrowfill@{\arrowfill@\equiv\equiv\Llongrightarrow}
\newcommand{\aMM}[2][]{%
	\ext@arrow 0099\xLlongrightarrowfill@{#1}{#2}}
\makeatother

\newcommand{\aoplaxn}[1]{\aR[\oplaxn]{#1}}
\newcommand{\aoplaxm}[1]{\aR[\oplaxm]{#1}}
\newcommand{\aloose}[1]{\aR[\opn{loose}]{#1}}

\newcommand{\iso}{\cong}

\newcommand{\aequi}{\ensuremath{\stackrel{\raisebox{-1ex}{\kern-.3ex$\scriptstyle\sim$}}{\rightarrow}}}
\newcommand{\aequii}{\ensuremath{\stackrel{\raisebox{-1ex}{\kern-.3ex$\scriptstyle\sim$}}{\longrightarrow}}}

\newcommand{\PB}[1]{\arrow[#1,phantom,"\scalebox{1.6}{\color{black}$\lrcorner$}",very near start]}
\newcommand{\Ar}[4][]{\arrow[#2,"{#3}"{#1},""{name=#4, anchor=center}]}
\newcommand{\Ars}[4][]{\arrow[#2,"{#3}"'{#1},""{name=#4, anchor=center}]}
\newcommand{\Arb}[6][]{\arrow[#2,"{#3}"{#1},from=#4,to=#5,shorten <= #6 em, shorten >= #6 em]}
\newcommand{\Arbs}[6][]{\arrow[#2,"{#3}"'{#1},from=#4,to=#5,shorten <= #6 em, shorten >= #6 em]}

\NewDocumentCommand{\fib}{O{n} O{2.3} mmm}{%
	\begin{cd}*[#2][5]
		{#3}\ifx#1n{\arrow[d,"{\,\scaleu{#4}}"]}\else{\ifx#1i{\arrow[d,hookrightarrow,"{\,\scaleu{#4}}"]}\else{\ifx#1e{\arrow[d,equal,"{\,\scaleu{#4}}"]}\else{\ifx#1R{\arrow[d,Rightarrow,"{\,\scaleu{#4}}"]}\fi}\fi}\fi}\fi\\
		{#5}\ifx#1o{\arrow[u,"{\,\scaleu{#4}}"']}\fi
	\end{cd}\xspace
}
\NewDocumentCommand{\fibdiag}{O{n} O{2.3} mmm}{%
	\begin{cd}*[#2][5]
		{#3}\ifx#1n{\arrow[d,"{\,{#4}}"]}\else{\ifx#1i{\arrow[d,hookrightarrow,"{\,{#4}}"]}\else{\ifx#1e{\arrow[d,equal,"{\,{#4}}"]}\else{\ifx#1R{\arrow[d,Rightarrow,"{\,{#4}}"]}\fi}\fi}\fi}\fi\\
		{#5}\ifx#1o{\arrow[u,"{\,{#4}}"']}\fi
	\end{cd}\xspace
}
	
\NewDocumentCommand{\sq}{s O{n} O{7} O{7} O{} O{2.7} O{2.2} O{0.5} O{n}}{%
	\def\foosq##1##2##3##4##5##6##7##8{%
		\IfBooleanTF{#1}{\begin{cd}*}{\begin{cd}}[#3][#4]
				{##1}\ifx#2p{\PB{rd}}\fi\arrow[r,"{##5}"]\ifx#9l{\arrow[d,equal,"{##6}"']}\else{\arrow[d,"{##6}"']}\fi\&{##2}\ifx#9r{\arrow[d,equal,"{##7}"]}\else{\arrow[d,"{##7}"]}\fi\ifx#2l{\arrow[ld,Rightarrow,shorten <=#6ex,shorten >=#7ex,"{#5}"{pos=#8}]}\fi\\
				{##3}\ifx#9d{\arrow[r,equal,"{##8}"']}\else{\arrow[r,"{##8}"']}\fi\ifx#2o{\arrow[ur,Rightarrow,shorten <=#6ex,shorten >=#7ex,"{#5}"{pos=#8}]}\fi\&{##4}
		\end{cd}}%
		\foosq}
\NewDocumentCommand{\tr}{s O{4.5} O{6.5} O{0} O{0} O{n} O{0} O{} O{0}}{%
	\def\footr##1##2##3##4##5##6{%
		\IfBooleanTF{#1}{\begin{cd}*}{\begin{cd}}[#3][#2]
				{##1}\arrow[rr,"{##4}"]
				\Ars[inner sep =0.2ex]{dr}{##5}{A}\&[#4ex]\&[#5ex]{##2}\Ar[inner sep =0.2ex]{ld}{##6}{B}\\
				\&{##3}
				\ifx#6l{\Arb{Rightarrow,shift right=#7em}{#8}{A}{B}{#9}}\else{\ifx#6o{\Arbs{Rightarrow,shift right=#7em}{#8}{B}{A}{#9}}\else{\ifx#6i{\Arbs[inner sep=0.9ex]{iso,shift right=#7em}{#8}{A}{B}{#9}}\else{\ifx#6e{\Arb{equal,shift right=#7em}{#8}{A}{B}{#9}}\else{}\fi}\fi}\fi}\fi
		\end{cd}}%
		\footr}
\NewDocumentCommand{\tc}{s t+ O{7} O{30} O{} O{} O{} o}{
	\def\footc##1##2##3##4##5{%
		\FPmul\Mulresulttwo{#3}{#3}%
		\FPmul\Mulresult{0.0026}{\Mulresulttwo}%
		\IfBooleanTF{#1}{\begin{cd}*}{\begin{cd}}[#3][#3]
				{##1}\Ar[#5]{r,bend left=#4}{##3}{A}\Ars[#6]{r,bend right=#4}{##4}{B}\&{##2}
				\IfBooleanTF{#2}{\Arb[description,pos=0.49]}{\Arb}{Rightarrow #7}{\mkern1mu {##5}}{A}{B}{\IfNoValueTF{#8}{\Mulresult}{#8}}
		\end{cd}}%
		\footc}

\begin{document}

\title{Colimits in 2-dimensional slices}
\author[L. Mesiti]{Luca Mesiti}
\address{School of Mathematics, University of Leeds}
\email{mmlme@leeds.ac.uk}
%\copyrightyear{2023}

\begin{abstract}
	We generalize to dimension 2 the well-known fact that a colimit in a 1-dimensional slice is precisely the map from the colimit of the domains of the diagram that is induced by the universal property. For this, we find the need to reduce weighted 2-colimits to cartesian-marked oplax conical ones, and as a consequence the need to consider lax slices.

	We prove results of preservation, reflection and lifting of 2-colimits for the domain 2-functor from a lax slice. We thus generalize to dimension 2 the whole fruitful calculus of colimits in 1-dimensional slices. We achieve this within the framework of enhanced (or $\F$-)category theory. The preservation result assumes products in the base 2-category and uses an original general theorem which states that a lax left adjoint preserves appropriate 2-colimits if the adjunction is strict on one side and suitably $\F$-categorical.
	
	Finally, we apply the same general theorem of preservation of 2-colimits to the 2-functor of change of base along a split Grothendieck opfibration between lax slices. We prove that this change of base 2-functor is indeed a left adjoint of the kind described above by laxifying the proof that Conduch\'{e} functors are exponentiable. We conclude extending the result of preservation of 2-colimits for the change of base 2-functor to any finitely complete 2-category with a dense generator.
\end{abstract}

\maketitle

\setcounter{tocdepth}{1}
\tableofcontents

\section{Introduction}

It is well-known that, in dimension $1$, a colimit in a slice is precisely the map from the colimit of the domains of the diagram that is induced by the universal property of the colimit. This fact, together with the results of preservation, reflection and lifting of all colimits for the domain functor from a slice, gives a complete calculus of colimits in $1$-dimensional slices (see \thex\ref{colimitsinslicesdim1}). And such a calculus has been proven useful in myriads of applications, in particular in the context of locally cartesian closed categories or for general exponentiability of morphisms, in categorical logic, algebraic geometry and topos theory. Indeed, an exponentiable morphism $f:E\to B$ in $\C$, that is an exponentiable object in $\slice{\C}{B}$, can be characterized as a morphism which admits all pullbacks along it and is such that the change of base functor $f\st\:\slice{\C}{B}\to \slice{\C}{E}$ has a right adjoint. The latter condition implies, and by adjoint functor theorems is often implied by, preservation of all colimits for $f\st$. The calculus of colimits in $1$-dimensional slices is what allows to apply the exponentiability of a morphism to colimits in the slice that come from colimits in the category $\C$, i.e.\ the ones that we have in practice.

The main result of this paper is a generalization to dimension $2$ of this fruitful $1$-dimensional calculus, including results of preservation, reflection and lifting of $2$-colimits for the domain $2$-functor from a lax slice. Theorems on suitable change of base $2$-functors between lax slices are presented as well. The lax slice is indeed the appropriate $2$-dimensional slice to consider in order to achieve such generalization, as we justify with two different approaches.

These results, in combination with our~\cite{mesiti_twosetenrgrothconstrlaxnlimits} and our joint work with Caviglia~\cite{cavigliamesiti_indexedgrothconstr}, are then applied in our~\cite{mesiti_twoclassifiersdensegenstacks} to produce a 2-dimensional subobject classifier in the 2-categories of stacks. This generalizes to dimension 2 the fundamental fact that Grothendieck topoi are elementary topoi.

The following theorem condenses the calculus of colimits in $1$-dimensional slices.
\begin{teor}\label{colimitsinslicesdim1}
	Let $\C$ be a category with products and let $M\in\C$. The domain functor $\dom\:\slice{\C}{M}\to\C$ preserves, reflects and lifts uniquely all colimits \pteor{and thus creates all colimits}.
	
	Moreover, for every diagram $D\:\A\to \C$ with $\A$ small that admits a colimit in $\C$, every morphism $q\:\colim_A{D(A)}\to M$ in $\C$ is the colimit of a diagram in $\slice{\C}{M}$. More precisely,
	\begin{equation}\label{propertycoliminslicedim1}
		\fib{\colim_A{D(A)}}{q}{M}=\h[3]\colim_{A}{\fib{D(A)}{q\c i_A}{M}}\quad\text{in }\slice{\C}{M},
	\end{equation}
	where the $i_A\:D(A)\to \colim_A{D(A)}$ are the inclusions that form the universal cocone.
\end{teor}
Notice that this theorem recovers the property we mentioned above, that a colimit in $\slice{\C}{M}$ is precisely the map from the colimit of the domains which is induced by the universal property. Indeed, half of this fact is captured by the preservation of colimits for $\dom\:\slice{\C}{M}\to \C$, whereas the other half, that is harder to capture, is represented by equation~\refs{propertycoliminslicedim1}. In dimension $1$, the latter special property holds because a cocone on $M$ is the same thing as a diagram in the slice over $M$. But in dimension $2$, we need weighted $2$-colimits and then weighted $2$-cocylinders rather than cocones. This makes it harder to establish a bijection with diagrams in a $2$-dimensional slice, as such diagrams still have a conical shape.

In this paper, we first focus on generalizing the special property of equation~\refs{propertycoliminslicedim1} to dimension $2$ (\thex\ref{teordomis2colimfib}), extracting from a now weighted $2$-cocylinder a diagram in a $2$-dimensional slice that works. We show two different approaches to this. The first approach (\conx\ref{consfirstapproach}) is more intuitive, and is based on the reduction of weighted $2$-colimits to essentially conical ones, namely cartesian-marked oplax conical ones. The idea of such reduction goes back to Street's~\cite{street_limitsindexedbycatvalued}. More recently, Descotte, Dubuc and Szyld brought new attention to it in their~\cite{descottedubucszyld_sigmalimandflatpseudofun}, where they call sigma colimits the pseudo version of cartesian-marked oplax conical colimits. We have described such reduction in detail with new more elementary proofs in~\cite{mesiti_twosetenrgrothconstrlaxnlimits}. The idea behind cartesian-marked oplax conical colimits is the following. Conical colimits do not suffice anymore in dimension $2$, basically because functors from $\1$ to a category $\C$ cannot capture the whole of $\C$, but just its objects. The philosophy behind weighted $2$-colimits is to capture the whole of $\C$ with functors from every possible category (or actually just $\2$) to it. But another solution is given by considering functors from $\1$ to $\C$ together with natural transformations between them. This brings to cartesian-marked oplax conical $2$-colimits, that are as expressive as weighted $2$-colimits, but offer substantial advantages in certain situations. Indeed there is sometimes the need, as in this paper, to use essentially conical shapes. The price to pay is to have (coherent) $2$-cells inside the cones. The reduction of weighted $2$-colimits to cartesian-marked oplax conical ones is allowed and regulated by the 2-category of elements construction. Such construction is an extension of the usual Grothendieck construction that admits $2$-functors $\A\to \CAT$ from a $2$-category $\A$ (rather than just a $1$-category). It was introduced by Street in~\cite{street_limitsindexedbycatvalued}; we studied this construction in detail in our~\cite{mesiti_twosetenrgrothconstrlaxnlimits}, both from an elementary and an abstract perspective.

The second approach, culminating with \thex\ref{teordomis2colimfib}, is instead more abstract, based on an apparently original concept of colim-fibration, that we give both in dimension $1$ (\defx\ref{defcolimfibrationdim1}) and dimension $2$ (\defx\ref{def2colimfib}), as well as on the 2-category of elements construction. Both the approaches show the need to consider lax slices in order to generalize the special property of equation~\refs{propertycoliminslicedim1} to dimension $2$. In the first one, for example, this corresponds to only being able to essentially conicalize weighted $2$-colimits, rather than to strictly conicalize them.

We also characterize (\thex\ref{teorcolimitsinlaxslicesthatcomefromthebase}) which colimits in the lax slice come from colimits in the base 2-category via the generalization of equation~\refs{propertycoliminslicedim1}. We explain how these are not all colimits, as opposed to what happens in dimension 1.

We achieve a result of reflection of $2$-colimits for the domain $2$-functor $\dom\:\laxslice{\E}{M}\to \E$ from a lax slice (\thex\ref{teordomreflects}). This is the main part of the proof that $\dom\:\laxslice{\E}{M}\to \E$ is a 2-colim-fibration (\thex\ref{teordomis2colimfib}). Further than reflecting appropriate $2$-colimits, a $2$-colim-fibration is in particular a discrete 2-fibration, that is, what is classified by the 2-category of elements construction. The notion of discrete $2$-fibration is described both in our~\cite{mesiti_twosetenrgrothconstrlaxnlimits} and in Lambert's~\cite{lambert_discretetwofib}. It corresponds to a locally discrete version of Hermida's~\cite{hermida_somepropoffibasafibredtwocat} $2$-fibrations.

We then show a result of lifting of $2$-colimits for $\dom\:\laxslice{\E}{M}\to \E$ in \prox\ref{proplifting}. This is based on a generalization to dimension $2$ of the bijective correspondence between cocones on $M$ and diagrams in the slice over $M$ (\prox\ref{correspoplaxncoconesdiag}). The $2$-dimensional correspondence is captured and justified by $\F$-category theory, also called enhanced $2$-category theory and introduced in Lack and Shulman's~\cite{lackshulman_enhancedtwocatlimlaxmor}. Roughly, the idea is to consider $2$-categories, whose morphisms we think as loose, with a selected subclass of morphisms that we call tight. We then ask $2$-functors to preserve the tightness of morphisms. A more thorough recall is in \recx\ref{recFcategorytheory}.

$\F$-category theory is then crucial in establishing a result of preservation of $2$-colimits for $\dom\:\laxslice{\E}{M}\to \E$ (\thex\ref{teordomhassemilaxFadjoint}), assuming that $\E$ has products. Indeed we guarantee preservation of $2$-colimits by an original general theorem (\thex\ref{teorsemilaxFadjpreserve}) that states that a lax left adjoint (\defx\ref{deflaxadj}) preserves appropriate 2-colimits (\defx\ref{defsoplaxFcolimit}) if the adjunction is strict on one side and suitably $\F$-categorical (see the partly original \defx\ref{deflaxFadj}). Lax adjoints have been firstly introduced in Gray's~\cite{gray_formalcattheory}, and the idea is to admit unit and counit to only be lax natural. At the level of hom-categories, such laxness translates as having an adjunction between them rather than an isomorphism (or an equivalence as for a biadjunction); see \recx\ref{reclaxadjunction}. We need lax adjunctions because the suitable right adjoint to the domain $2$-functor $\dom\:\laxslice{\E}{M}\to \E$ is still $M\x -$ (as in dimension $1$), but the unit is now only lax natural. A lax adjunction between $\dom\:\laxslice{\E}{M}\to \E$ and $M\x -$ would not be enough to guarantee that $\dom$ preserves 2-colimits. But taking then the tight part of a lax slice to be the strict slice (see \remx\ref{rephrasecorrespondence}), we prove the crucial fact that such a lax adjunction is suitably $\F$-categorical and strict on one side. And we can then conclude that $\dom$ preserves appropriate 2-colimits. $\F$-categorical adjunctions appear also in Walker's~\cite{walker_laxfamilialrepreslaxgenfact}; another potential source of examples is Bourke's~\cite{bourke_twodimmonadicity}. Remember that, despite the preservation result is expressed in an $\F$-categorical language, it is always possible to start from a weighted $2$-colimit and view it in this context, after reducing it to a cartesian-marked oplax conical one (\recx\ref{rectwocatofel}). We also explain what such $\F$-categorical assumptions explicitly mean in practice.

Finally, we apply the general $\F$-categorical theorem of preservation of $2$-colimits to the $2$-functor of change of base along a split Grothendieck opfibration between lax slices (see \prox\ref{proptaustextendstoatwofunctor}). In dimension $1$, the concept of change of base between slices is definitely helpful, and it is well-known that the pullback perfectly realizes such a job. For $\CAT$, given a functor $\tau\:\E\to \B$, it is still a good idea to consider the pullback $2$-functor $\tau\st\:\slice{\CAT}{\B}\to \slice{\CAT}{\E}$ between strict slices. It is well-known that such change of base $2$-functor has a right 2-adjoint $\tau\stb$, and thus preserves all weighted 2-colimits, precisely when $\tau$ is a Conduch\'{e} functor (as Conduch\'{e} functors are the exponentiable morphisms in $\CAT$). This is proved in Conduch\'{e}'s~\cite{conduche_existadjointsadroitcondfunc}, with the ideas already present in Giraud's~\cite{giraud_methodedeladescente}. However, in order to generalize the calculus of colimits in $1$-dimensional slices to dimension $2$, we need to consider lax slices. And it is then helpful to have a change of base $2$-functor between lax slices of a finitely complete $2$-category. We can achieve this by calculating comma objects or, equivalently, pullbacks along split Grothendieck opfibrations (that serve as a kind of fibrant replacement, see \prox\ref{propreplacement}). Taking pullbacks along opfibrations is preferable in the context of this paper, since Grothendieck opfibrations in $\CAT$ are always Conduch\'{e} and we can generalize the ideas for finding a right adjoint to the pullback functor $\tau\st\:\slice{\CAT}{\B}\to \slice{\CAT}{\E}$ (from Conduch\'{e}'s~\cite{conduche_existadjointsadroitcondfunc}) to lax slices. Notice that lax slices have a fixed direction and we then need to take opfibrations, while Conduch\'{e} functors are now too unbiased.

We prove that $\tau\st\:\laxslice{\CAT}{\B}\to \laxslice{\CAT}{\E}$ has a loose $\F$-categorical lax right adjoint that is strict on one side, in \thex\ref{teortausthasadjoint}. This generalizes also Palmgren's~\cite{palmgren_groupoidsloccartclos}, that proved a similar result for pseudoslices of groupoids, from the comma objects point of view. It is also related to Johnstone's~\cite{johnstone_fibrationsandpartialproducts}, which characterizes Grothendieck opfibrations as those functors for which postcomposition gives a lax left adjoint to $\tau\st\:\laxslice{\CAT}{\B}\to \laxslice{\CAT}{\E}$ that is (pseudo-)strict on one side. Our theorem then implies the preservation of appropriate $2$-colimits for the $2$-functor $\tau\st$ between lax slices.

We also show that the $2$-functor of change of base along a split Grothendieck opfibration between lax slices makes sense in a general $2$-category rather than just in \CAT (\prox\ref{proptaustextendstoatwofunctor}). We conclude proving that the $2$-functor $\tau\st$ between lax slices preserves appropriate $2$-colimits also in the case of prestacks (\prox\ref{proptaustprestacks}) and more in general of finitely complete $2$-categories with a dense generator (\thex\ref{teortaustdensegenerator}).

\subsection*{Outline of the paper}

In Section~\ref{sectionlaxfadjpreserv}, we prove a general theorem (\thex\ref{teorsemilaxFadjpreserve}) which states that a lax left adjoint preserves appropriate colimits if the adjunction is strict on one side and suitably $\F$-categorical. We also introduce a few needed original $\F$-categorical notions, after recalling the basics of $\F$-category theory and lax adjunctions. The results of this section will be very useful in Sections~\ref{sectionliftingpreservfordom} and~\ref{sectionchangeofbase}. 

In Section~\ref{sectioncolimitsin2slices}, we generalize to dimension 2 the well-known fact that a colimit in a 1-dimensional slice is the map from the colimit of the domains of the diagram that is induced by the universal property (\thex\ref{teordomis2colimfib}). We do this with two different approaches: firstly via reducing weighted $2$-colimits to cartesian-marked oplax conical ones (\conx\ref{consfirstapproach}) and then via $2$-colim-fibrations (\defx\ref{def2colimfib}). This section includes a result of reflection of $2$-colimits for $\dom\:\laxslice{\E}{M}\to \E$. 

In Section~\ref{sectionliftingpreservfordom}, we generalize to dimension $2$ the bijective correspondence between cocones on $M$ and diagrams in the slice over $M$ (\prox\ref{correspoplaxncoconesdiag}), thanks to $\F$-category theory. We then prove results of lifting and preservation of $2$-colimits for $\dom\:\laxslice{\E}{M}\to \E$ (\prox\ref{proplifting} and \thex\ref{teordomhassemilaxFadjoint}). We achieve the result of preservation via the work of Section~\ref{sectionlaxfadjpreserv}, showing that $\dom\:\laxslice{\E}{M}\to \E$ has an appropriate $\F$-categorical lax right adjoint (when $\E$ has products). Finally, we characterize which 2-colimits in the lax slice come from 2-colimits of the base via the work of Section~\ref{sectioncolimitsin2slices}.

In Section~\ref{sectionchangeofbase}, we apply the general theorem of preservation of $2$-colimits presented in Section~\ref{sectionlaxfadjpreserv} to the $2$-functor of change of base along a split Grothendieck opfibration between lax slices (\thex\ref{teortausthasadjoint}). We show indeed that such change of base 2-functor has an appropriate $\F$-categorical lax right adjoint. We conclude extending the result of preservation of 2-colimits for the change of base between lax slices to prestacks (\prox\ref{proptaustprestacks}) and then to any finitely complete $2$-category with a dense generator (\thex\ref{teortaustdensegenerator}).

\section{Lax $\F$-adjoints and preservation of colimits}\label{sectionlaxfadjpreserv}

In this section, we present a general theorem of preservation of 2-colimits for lax left adjoints (\thex\ref{teorsemilaxFadjpreserve}). We achieve this within the framework of enhanced (or \mbox{$\F$-)}category theory, introduced in Lack and Shulman's~\cite{lackshulman_enhancedtwocatlimlaxmor}. This theorem will be very useful in the following sections to prove results of preservation of 2-colimits for both the domain 2-functor from a lax slice (assuming to have products in the base 2-category) (\thex\ref{teordomhassemilaxFadjoint}) and the change of base 2-functors between lax slices (\thex\ref{teortausthasadjoint}). We will indeed explain in Section~\ref{sectioncolimitsin2slices} that we need to consider lax slices in order to generalize to dimension 2 the fruitful calculus of colimits in 1-dimensional slices. And taking lax slices will result in having lax adjunctions rather than strict ones. For example, assuming to have products in $\E$, the domain 2-functor $\dom\:\laxslice{\E}{M}\to \E$ has only a lax right adjoint, given by $M\x -$ (\thex\ref{teordomhassemilaxFadjoint}).

The work of Section~\ref{sectioncolimitsin2slices}, which includes the generalization to dimension 2 of the special property of equation~\refs{propertycoliminslicedim1} (of \thex\ref{colimitsinslicesdim1}), can be understood also without reading this section. Nonetheless, $\F$-category theory will be helpful in the entire paper.

We begin recalling basic definitions and tools of $\F$-category theory, for which we take as main reference Lack and Shulman's~\cite{lackshulman_enhancedtwocatlimlaxmor}. We also recall the concept of lax adjunction and the universal mapping property that characterizes it, for which we take as references Gray's~\cite{gray_formalcattheory} and Bunge's~\cite{bunge_coherentextrelationalalg}. We then aim at the main theorem of this section, introducing also needed original notions in the framework of $\F$-category theory. Although being a lax left adjoint is not enough to preserve 2-colimits, we prove (\thex\ref{teorsemilaxFadjpreserve}) that being a \predfn{right-semi-lax left $\F$-adjoint} is enough to preserve all \predfn{tight strict/oplax $\F$-colimits}. Furthermore, we prove that being only a \predfn{right-semi-lax loose left $\F$-adjoint} is enough to guarantee the preservation of a large class of $2$-colimits. 

\begin{rec}\label{recFcategorytheory}
	$\F$ is the cartesian closed full subcategory of $\CAT^\2$ (the category of arrows in $\CAT$) determined by the functors which are injective on objects and fully faithful (i.e.\ full embeddings). It is possible to enrich over $\F$, obtaining $\F$-category theory. An $\F$-category $\S$ is then given by a collection of objects, a hom-category ${\HomC{\S}{X}{Y}}\tight$ of tight morphisms and a second hom-category ${\HomC{\S}{X}{Y}}\loose$ of loose morphisms that give $2$-category structures (respectively) $\S\tight$ and $\S\loose$ to $\S$, together with an identity on objects, faithful and locally fully faithful $2$-functor $J_\S\:\S\tight\to\S\loose$. An $\F$-functor $F\:\S\to \T$ is a $2$-functor $F\loose\:\S\loose\to\T\loose$ that restricts to a $2$-functor $F\tight\:\S\tight\to\T\tight$ (forming a commutative square); this is equivalent to $F\loose$ preserving tightness. And an $\F$-natural transformation is a $2$-natural transformation $\alpha\loose$ between loose parts that restricts to one between the tight parts; this is equivalent to $\alpha\loose$ having tight components. It is then true that the category $\F$ is enriched over itself, with tight morphisms the morphisms of $\F$, loose morphisms the functors between loose parts and $2$-cells the $2$-natural transformations between the latter. And for every $\F$-category $\S$ and $S\in \S$ we can build a copresheaf $\HomC{\S}{S}{-}\:\S\to \F$, that sends $S'$ to the full embedding $\HomC{\S\tight}{S}{S'}\to\HomC{\S\loose}{S}{S'}$.
	
	Given $\F$-categories $\S$ and $\T$, there is an $\F$-category $\mF{\S}{\T}$ of $\F$-functors from $\S$ to $\T$, where the tight morphisms are the $\F$-natural transformations, the loose morphisms are the $2$-natural transformations between the loose parts and the $2$-cells are the modifications between the loose morphisms. But we will need also an (op)lax version $\mFopolax{\S}{\T}$ of it, which is the $\F$-category defined as follows:
	\begin{description}
		\item[an object] is an $\F$-functor $G\:\S\to\T$;
		\item[a loose morphism $G\aloose{}H$] is an (op)lax natural transformation $\alpha\loose$ between the loose parts such that the structure $2$-cells on tight morphisms are identities, that precisely means that $\alpha\loose\ast J_\S$ is (strictly) $2$-natural; we call them \dfn{loose strict/(op)lax};
		\item[a tight morphism $G\aR{}H$] is a loose one that restricts to a $2$-natural transformation between the tight parts, which is equivalent to a loose morphism with tight components; they are usually called \dfn{strict/(op)lax};
		\item[a $2$-cell] is a modification between the loose morphisms.
	\end{description}
	We can then apply the definitions above to the case $\T=\F$, obtaining two $\F$-categories of copresheaves on $\S$. The strict one, $\mF{\S}{\F}$, can be characterized as follows:
	\begin{description}
		\item[an object] is an $\F$-functor $G\:\S\to\F$, that we can identify with a pair of $2$-functors $G\tight\:\S\tight\to \CAT$ and $G\loose\:\S\loose\to \CAT$ together with a $2$-natural transformation
		\tr[2.8][4][0][0][l][-0.235][j_G][0.85]{\S\tight}{\S\loose}{\CAT}{J_\S}{G\tight}{G\loose}
		whose components are all full embeddings;
		\item[a loose morphism $G\aloose{}H$] is a $2$-natural transformation $\alpha\loose\:G\loose\aR{}H\loose\:\S\loose\to \CAT$;
		\item[a tight morphism $G\aR{}H$] is a loose one with tight components, that precisely means that it induces a $2$-natural transformation $\alpha\tight\:G\tight\aR{}H\tight$ such that
		\begin{eqD*}
		\begin{cd}*
			\S\tight \arrow[r,"{J_\S}"]\arrow[rd,bend right,"{G\tight}"'{inner sep=0.2ex},""{name=A}]\& \S\loose \arrow[d,bend right=40,"{G\loose}"'{pos=0.67,inner sep=0.2ex},""{name=B}]\arrow[d,bend left=43,"{H\loose}",""'{name=C}]\arrow[Rightarrow,from=A,"{j_G}"{inner sep=0.3ex,pos=0.4},shorten <=0.7ex,shorten >=2.5ex]\\
			\& \CAT
			\Arb[pos=0.45]{Rightarrow}{\alpha\loose}{B}{C}{0.2}
		\end{cd}\h[4]=\h[4]
		\begin{cd}*
			\S\tight \arrow[r,"{J_\S}"]\arrow[d,bend right=43,"{G\tight}"'{inner sep=0.35ex},""{name=A}]\arrow[d,bend left=40,"{H\tight}"{pos=0.65,inner sep=0.2ex},""'{name=B}]\& |[alias=K]|\S\loose \arrow[ld,bend left,"{H\loose}",""'{name=C}]\\
			\CAT
			\Arb[pos=0.48]{Rightarrow}{\alpha\tight}{A}{B}{0.2}
			\arrow[Rightarrow,from=B,to=K,"{j_H}"'{pos=0.37,inner sep=0.3ex},shorten <=1.35ex,shorten >=2.2ex]
		\end{cd}
		\end{eqD*}
		\item[a $2$-cell] is a modification between the loose morphisms.
	\end{description}
	
	Whereas the oplax version $\mFoplax{\S}{\F}$ can be characterized as follows:
	\begin{description}
		\item[an object] is an object of $\mF{\S}{\F}$, that we keep on viewing as a triangle above (in the description of $\mF{\S}{\F}$);
		\item[a loose morphism $G\aloose{}H$] is an oplax natural transformation $\alpha\loose\:G\loose\to H\loose$ that is (strictly) $2$-natural on tight morphisms, meaning that $\alpha\loose\ast J_\S$ is $2$-natural; we call them \dfn{marked oplax};
		\item[a tight morphism $G\aR{}H$] is a loose one with tight components, that precisely means that it induces a $2$-natural transformation $\alpha\tight\:G\tight\aR{}H\tight$ such that
		\begin{eqD*}
			\begin{cd}*
				\S\tight \arrow[r,"{J_\S}"]\arrow[rd,bend right,"{G\tight}"'{inner sep=0.2ex},""{name=A}]\& \S\loose \arrow[d,bend right=40,"{G\loose}"'{pos=0.67,inner sep=0.2ex},""{name=B}]\arrow[d,bend left=43,"{H\loose}",""'{name=C}]\arrow[Rightarrow,from=A,"{j_G}"{inner sep=0.3ex,pos=0.4},shorten <=0.7ex,shorten >=2.5ex]\\
				\& \CAT
				\Arb[pos=0.45]{Rightarrow}{\alpha\loose}{B}{C}{0.2}
			\end{cd}\h[4]=\h[4]
			\begin{cd}*
				\S\tight \arrow[r,"{J_\S}"]\arrow[d,bend right=43,"{G\tight}"'{inner sep=0.35ex},""{name=A}]\arrow[d,bend left=40,"{H\tight}"{pos=0.65,inner sep=0.2ex},""'{name=B}]\& |[alias=K]|\S\loose \arrow[ld,bend left,"{H\loose}",""'{name=C}]\\
				\CAT
				\Arb[pos=0.48]{Rightarrow}{\alpha\tight}{A}{B}{0.2}
				\arrow[Rightarrow,from=B,to=K,"{j_H}"'{pos=0.37,inner sep=0.3ex},shorten <=1.35ex,shorten >=2.2ex]
			\end{cd}
		\end{eqD*}
		\item[a $2$-cell] is modification between the loose morphisms.
	\end{description}
\end{rec}

\begin{defne}\label{deflaxadj}
	A lax adjunction is, for us, what Gray calls a strict weak quasi-adjunction in~\cite{gray_formalcattheory}. That is, a \dfn{lax adjunction} from a $2$-functor $F\:\A\to \B$ to a $2$-functor $U\:\B\to \A$ is given by a lax natural unit $\eta\:\Id{}\aR{} U\c F$, a lax natural counit $\eps\:F\c U\aR{} \Id{}$ and modifications
	\begin{eqD*}
	\begin{cd}*[4][4]
		F \arrow[rr,equal,"{}"] \arrow[rd,"{F\h\eta}"'{inner sep=0.2ex}]\&\hphantom{.}\arrow[d,Rightarrow,"{s}"{pos=0.38},shorten <= 0ex, shorten >=0.7ex]\& F \\
		\& F\c U\c F\arrow[ru,"{\eps\h F}"'{inner sep=0.2ex}]
	\end{cd}\quad\quad
	\begin{cd}*[4][4]
		\& U\c F\c U\arrow[rd,"{U\h\eps}"{inner sep=0.2ex}] \arrow[d,Rightarrow,"{t}"{pos=0.47},shorten <= 1.1ex,shorten >=1ex] \\
		U \arrow[rr,equal,"{}"] \arrow[ru,"{\eta\h U}"{inner sep=0.2ex}]\&\hphantom{.}\& U
	\end{cd}
	\end{eqD*}
	that express lax triangular laws, such that both the swallowtail modifications
	\begin{eqD*}
		\begin{cd}*[5.3][5.3]
			\id{} \arrow[d,"{\eta}"']\arrow[r,"{\eta}"]\& UF \arrow[ld,Rightarrow,"{\eta_\eta}",shorten <= 2.2ex,shorten >=1.6ex] \arrow[d,"{U\h[-1]F \eta}"]\arrow[rdd,equal,bend left,"{}"'{name=A}]\\
			UF \arrow[rrd,equal,bend right=29,"{}"{name=B}]\arrow[r,"{\eta U\h[-1]F}"'] \& UFUF\arrow[rd,"{U\eps F}"{inner sep=0.2ex}]\arrow[Rightarrow,from=A,"{Us}"{pos=0.65},shorten <=0.45ex,shorten >=-0.1ex]\arrow[Rightarrow,to = B,"{tF}"{pos=0.35},shorten <=0.9ex,shorten >=0.5ex] \\[-4ex]
			\& \&[-1ex] UF
		\end{cd}\quad
		\begin{cd}*[5.3][5.3]
			FU \arrow[rd,"{F\eta U}"{inner sep=0.2ex}]\arrow[rdd,equal,bend right,""{name=B}]\arrow[rrd,equal,bend left,"{}"'{name=A}]\&[-1ex] \\[-3.8ex]
			\& FUFU \arrow[Rightarrow,from=A,"{sU}"{pos=0.35},shorten <=0.75ex,shorten >=0.5ex]\arrow[Rightarrow,to = B,"{Ft}"'{pos=0.35},shorten <=0.4ex,shorten >=0ex] \arrow[r,"{\eps FU}"] \arrow[d,"{FU\eps}"']\& FU\arrow[d,"{\eps}"]\arrow[ld,Rightarrow,"{\eps_\eps}",shorten <= 2.2ex,shorten >=1.6ex]\\
			\& FU \arrow[r,"{\eps}"']\& \id{}
		\end{cd}
	\end{eqD*}
	 are identities.
	 
	 A \dfn{right-semi-lax adjunction} is a lax adjunction in which the counit $\eps$ is strictly $2$-natural and the modification $s$ is the identity.
	 
	 We call a lax adjunction \dfn{strict} when $s$ and $t$ are both identities, making the triangular laws to hold strictly.
\end{defne}

\begin{rec}\label{reclaxadjunction}
	Using lax comma objects, firstly introduced in Gray's~\cite{gray_formalcattheory} and refined in our~\cite{mesiti_twosetenrgrothconstrlaxnlimits}, we can reduce the study of lax adjunctions to ordinary adjunctions between hom-sets. Indeed, according to Gray, a lax adjunction is equivalently given by homomorphic $2$-adjoint functors
	\begin{cd}
		\laxslice{F}{\B}\arrow[r,bend left,"{S}",""{name=A}]\& \laxslice{\A}{U}\arrow[l,bend left,"{T}",""{name=B}]
		\Arb{adjunction}{}{A}{B}{0}
	\end{cd}
	over $\A\x \B$ with unit $\chi\:\id{}\aR{} T\c S$ and counit $\xi\:S\c T\aR{}\id{}$ (that are automatically $2$-natural if assumed natural) over $\A\x \B$. Here $S$ and $T$ homomorphic means that they are given uniquely by lax natural $\eta$ and $\eps$ (it can be defined precisely as in 1.5.10 of Gray's~\cite{gray_formalcattheory}, asking for example $T$ to transform precomposition with cells in $\A$ into precomposition with the image through $F$ of those cells; see also below how we produce $T$ from $\eps$).
	
	Strictness corresponds to $$\chi\ast i_F=\id{} \quad\h[4]\text{ and }\quad\h[4] \xi\ast i_U=\id{},$$
	where $i_F\:\A\to \laxslice{F}{\B}$ is the $2$-functor induced by $\id{F}$ and analogously for $i_U$.
	
	Such a $2$-adjunction $S\dashv T$ means, in particular, that we have ordinary adjunctions between homsets
	 \begin{cd}[6][6]
	 	\HomC{\B}{F(A)}{B}\arrow[r,bend left,"{S}",""{name=A}]\& \HomC{\A}{A}{U(B)}\arrow[l,bend left,"{T}",""{name=B}]
	 	\Arb{adjunction}{}{A}{B}{0}
	 \end{cd}
 	for every $A\in \A$ and $B\in \B$. And we can rephrase such ordinary adjunctions in terms of having universal units. The global adjunction $S\dashv T$ corresponds, then, to such units satisfying a broader universal property, that captures the possibility of $h\:F(A)\to B$ to vary in the whole lax comma object $\laxslice{F}{\B}$ rather than in just $\HomC{\B}{F(A)}{B}$. This is the idea behind \prox\ref{univmappingproplaxadj}, that shows the universal mapping property that characterizes lax adjunctions. Before that, it is helpful to see explicitly how a lax adjunction $(F,U,\eta,\eps,s,t)$ produces the adjunctions $(S,T,\chi,\xi)$ between the homsets on $A\in\A$ and $B\in\B$, as we will use this later. $S$ and $T$ are defined as usual as
 	$$S=(-\c\eta_A)\c U \qquad \quad T=(\eps_B\c -)\c F$$
 	And the lax naturality of $\eta$ and $\eps$ gives $\chi$ and $\xi$; precisely, given $h\:F(A)\to B$ in $\B$ and $k\:A\to U(B)$ in $\A$\v[-1]
 	$$\chi_h=\h[4]\begin{cd}*[1.1][1.5]
 		\&[1.5ex]\& F(A)\arrow[rd,"{h}"]\arrow[dd,Rightarrow,"{\eps_h}",shorten <=1.4ex,shorten >=1.3ex]\\
 		F(A) \arrow[rru,equal,bend left=30,"{}"'{name=A}]\arrow[r,"{F(\eta_A)}"']\& F(U(F(A)))\arrow[Rightarrow,from=A,"{s_A}"{pos=0.7},shorten <=0.5ex,shorten >=0.4ex]\arrow[ru,"{\eps_{F(A)}}"{pos=0.63,inner sep=0.15ex}]\arrow[rd,"{F(U(h))}"'{pos=0.63,inner sep=0.2ex}]\&\& B\\
 		\&\&F(U(B))\arrow[ru,"{\eps_B}"']
 	\end{cd}\v[-1]$$
 	$$\xi_k=\h[4]\begin{cd}*[1.1][1.5]
 		\& U(F(A))\arrow[dd,Rightarrow,"{\eta_k}",shorten <=1.4ex,shorten >=1.3ex]\arrow[rd,"{U(F(k))}"{pos=0.37,inner sep=0.2ex}]\\
 		A \arrow[ru,"{\eta_A}"{inner sep=0.25ex}]\arrow[rd,"{k}"']\&\& |[alias=K]|U(F(U(B)))\arrow[r,"{U(\eps_B)}"] \&[1.5ex] U(B)\\
 		\& U(B) \arrow[ru,"{\eta_{U(B)}}"'{pos=0.37}]\arrow[rru,equal,bend right=26,"{}"{name=B}]
 		\arrow[Rightarrow,from=K,to=B,"{t_B}"{pos=0.73},shorten <=0.5ex,shorten >=0.4ex]
 	\end{cd}\v[1]$$
 
 	In particular, we see that for a right-semi-lax adjunction we obtain $\chi=\id{}$ and then $T\c S=\Id{}$ (this is where the name comes from).
\end{rec}

\begin{prop}[dual of Proposition~I,7.8.2 in Gray's~\cite{gray_formalcattheory} and of Theorem~4.1 in Bunge's~\cite{bunge_coherentextrelationalalg}]\label{univmappingproplaxadj}
	Let $F\:\A\to \B$ be a $2$-functor. Suppose that for every $B\in \B$ there is an object $U(B)\in \A$ and a morphism $\eps_B\:F(U(B))\to B$ in $\B$ that is universal in the following sense: for every $h\:F(A)\to B$ in $\B$ there is an $\ov{h}\:A\to U(B)$ in $\A$ and a $2$-cell
	\begin{cd}[6][8]
		F(A) \arrow[rd,bend left,"{h}",""'{name=C}]\arrow[d,"{F(\ov{h})}"'] \\
		F(U(B))\arrow[r,"{\eps_B}"'] \arrow[Rightarrow,from=C,"{\lambda_h}",shorten <=1ex,shorten >= 1.3ex]\& B
	\end{cd}
	in $\B$ such that, given any other $g\:A\to U(B)$ and $\sigma\:h\aR{}\eps_B\c F(g)$, there is a unique $\delta\:\ov{h}\aR{}g$ such that
	$$\begin{cd}*[6][8]
		F(A) \arrow[rd,bend left,"{h}",""'{name=C}]\arrow[d,bend left=20,"{F(\ov{h})}"{pos=0.44},""'{name=A}]\arrow[d,bend right=60,"{F(g)}"',""{name=B},shorten <=-0.8ex,shorten >=0.2ex] \\
		F(U(B))\arrow[r,"{\eps_B}"'] \arrow[Rightarrow,from=C,"{\lambda_h}",shorten <=1ex,shorten >= 1.3ex]\& B
		\Arb[inner sep=0.55ex]{Rightarrow,shift right=0.2ex}{F(\delta)}{A}{B}{0.2}
	\end{cd}=\h[4]\begin{cd}*[6][8]
		F(A) \arrow[rd,bend left,"{h}",""'{name=C}]\arrow[d,"{F(g)}"'] \\
		F(U(B))\arrow[r,"{\eps_B}"'] \arrow[Rightarrow,from=C,"{\sigma}",shorten <=1ex,shorten >= 1.3ex]\& B
	\end{cd}\v[2]$$
	Assume then that, for every $h\:F(A)\to B$ in $\B$, we have $\ov{h}=\ov{h\c\h\eps_{F(A)}}\c \ov{\id{F(A)}}$ and
	\begin{equation}\label{assumptioncharactlaxadj}\begin{cd}*[5][8]
		F(A)\arrow[rd,equal,bend left,"{}"'{name=A}] \arrow[d,"{F(\ov{\id{}})}"'] \\
		F(U(F(A))) \arrow[Rightarrow,from=A,"{\lambda_{\id{}}}",shorten <=1.3ex,shorten >= 1.6ex]\arrow[d,"{F(\ov{h\c\h \eps_{F(A)}})}"']\arrow[r,"{\eps_{F(A)}}"']\& F(A)\arrow[d,"{h}"] \arrow[ld,Rightarrow,"{\lambda_{h\c\h\eps_{F(A)}}}",shorten <=1.9ex,shorten >=1.6ex]\\
		F(U(B))\arrow[r,"{\eps_B}"']\& B		
	\end{cd}=\begin{cd}*[6][8]
		F(A) \arrow[rd,bend left,"{h}",""'{name=C}]\arrow[d,"{F(\ov{h})}"'] \\
		F(U(B))\arrow[r,"{\eps_B}"'] \arrow[Rightarrow,from=C,"{\lambda_h}",shorten <=1ex,shorten >= 1.3ex]\& B
	\end{cd}
	\end{equation}
	and also that, for every $B\in \B$, we have $\ov{\eps_B}=\id{}$ and $\lambda_{\eps_{B}}=\id{}$.
	
	Then $U$ extends to an oplax functor, $\eps$ extends to a lax natural transformation and there exist a lax natural transformation $\eta$ and modifications $s,t$ such that $U$ is a lax right-adjoint to $F$, except that in general $U$ is only an oplax functor \pteor{and the swallowtail identities need to be slightly modified accordingly}. 
	
	In particular, if $\lambda_h=\id{}$ for every $h\:F(A)\to B$, we obtain a right-semi-lax adjunction.
\end{prop}
\begin{proof}[Proof \pteor{constructions}]
	Given $g\:B\to B'$ in $\B$, define $U(g)\deq \ov{g\c \eps_B}$ and $\eps_{g}\deq \lambda_{g\c \h\eps_B}$. Given a $2$-cell $\mu\:g\aR{}g'$ in $\B$, define $U(\mu)$ as the unique $2$-cell induced by $\eps_{g'}\c \mu\h\eps_B$. Given composable morphisms $g$ and $g'$ in $\B$, pasting $\eps_{g}$ and $\eps_{g'}$ induces a unique coassociator for $U$, while the identity $2$-cell induces a unique counitor.
	
	We then define $\eta_A\deq \ov{\id{F(A)}}$ and $s_A\deq \lambda_{\id{F(A)}}$.
	\begin{cd}[6][8]
		F(A) \arrow[rd,bend left,equal,"{}",""'{name=C}]\arrow[d,"{F(\eta_A)}"'] \\
		F(U(F(A)))\arrow[r,"{\eps_{F(A)}}"'] \arrow[Rightarrow,from=C,"{s_A}",shorten <=1ex,shorten >= 1.3ex]\& F(A)
	\end{cd}
	And for every $f\:A\to A'$ in $\A$ we take $\eta_{f}$ to be the unique $2$-cell that is induced from
	\begin{cd}[5][8]
		F(A)\arrow[rd,equal,bend left,"{}"'{name=A}] \arrow[d,"{F(\eta_A)}"'] \\
		F(U(F(A))) \arrow[Rightarrow,from=A,"{s_A}",shorten <=1.3ex,shorten >= 1.6ex]\arrow[d,"{F(U(F(f)))}"']\arrow[r,"{\eps_{F(A)}}"']\& F(A)\arrow[d,"{F(f)}"] \arrow[ld,Rightarrow,"{\eps_{F(f)}}",shorten <=1.9ex,shorten >=1.6ex]\\
		F(U(F(A')))\arrow[r,"{\eps_{F(A')}}"']\& F(A')		
	\end{cd}
	considering $s_{A'}\ast F(f)$, thanks to the assumption in equation~\refs{assumptioncharactlaxadj}. Finally, for every $B\in \B$, we induce $t_B$ from the identity $2$-cell on $\eps_B$, using again the assumption in equation~\refs{assumptioncharactlaxadj}, with $h=\eps_B$.
\end{proof}

\begin{rem}
	In our two examples, i.e.\ with $F$ equal to $\dom\:\laxslice{\E}{M}\to\E$ (\thex\ref{teordomhassemilaxFadjoint}) and to the $2$-functor of pullback along a split Grothendieck opfibration between lax slices of $\CAT$ (\thex\ref{teortausthasadjoint}), \prox\ref{univmappingproplaxadj} will produce a strict right-semi-lax adjunction between $2$-functors. But this would not be enough to guarantee preservation of colimits. It is enough if the right-semi-lax adjunction is $\F$-categorical, as we prove in \thex\ref{teorsemilaxFadjpreserve}. But we have to restrict the attention to the \predfn{\pteor{tight} strict/oplax $\F$-colimits}, defined in \defx\ref{defsoplaxFcolimit} (that we believe are the suitable colimits to consider in this context).
	
	The concept of lax $\F$-adjunction appears in Walker's~\cite{walker_laxfamilialrepreslaxgenfact}, but in a pseudo/lax version and with the stronger request that $s$ and $t$ are isomorphisms. Moreover, only what for us is the tight version is considered there, asking the unit $\eta$ and the counit $\eps$ to be (tight) pseudo/oplax $\F$-natural rather than loose ones. The latter request means that $\eta$ and $\eps$ are tight morphisms in some $\mFoplax{\S}{\S}$ of \recx\ref{recFcategorytheory} rather than loose ones. Such request is not necessary to guarantee the preservation of the ``loose part" of tight strict/oplax $\F$-colimits. Moreover, it is not satisfied by our example of change of base along a split Grothendieck opfibration between lax slices.
\end{rem}

\begin{defne}\label{deflaxFadj}
	A \dfn{loose lax $\F$-adjunction} is a lax adjunction $(F,U,\eta,\eps,s,t)$ between the loose parts in which $F$ and $U$ are $\F$-functors and $\eta$ and $\eps$ are loose strict/lax $\F$-natural transformations (i.e.\ loose morphisms in $\mFlax{\S}{\S}$ of \recx\ref{recFcategorytheory} for suitable $\S$).\v[0.5]
	
	A \dfn{\pteor{tight} lax $\F$-adjunction} is a loose one such that $\eta$ and $\eps$ are (tight) strict/lax $\F$-natural transformations (that is, have tight components).
	
	A \dfn{right-semi-lax loose $\F$-adjunction} is a loose lax $\F$-adjunction such that $\eps$ is strictly $2$-natural (i.e.\ a loose morphism in $\mF{\S}{\S}$ of \recx\ref{recFcategorytheory}) and $s$ is the identity.\v[0.5]
	
	We call a loose lax $\F$-adjunction \dfn{strict} if both $s$ and $t$ are identities.
\end{defne}

\begin{defne}\label{defsoplaxFcolimit}
	Let $\A$ be a small $\F$-category and consider $\F$-functors $W\:\A\op\to \F$ (the weight) and $H\:\A\to \S$ (the $\F$-diagram). The \dfn{strict/oplax $\F$-colimit of $H$ weighted by $W$}, denoted as $\oplaxFcolim{W}{H}$, is (if it exists) an object $C\in \S$ together with an isomorphism in $\F$
	$$\HomC{\S}{C}{Q}\iso \HomC{\mFoplax{\A\op}{\F}}{W}{\HomC{\S}{H(-)}{Q}}$$
	$\F$-natural in $Q\in \S$,\v where $\mFoplax{\A\op}{\F}$ is the $\F$-category defined in \recx\ref{recFcategorytheory}.
\end{defne}
	
\begin{rem}\label{remunivcoconetightjoitnlydetect}
	The natural isomorphism of \defx\ref{defsoplaxFcolimit} is equivalently a $2$-natural isomorphism between the loose parts, that is,
	\begin{equation}\label{eqstrictoplaxFcolim}
		\HomC{\S\loose}{C}{Q}\iso \HomC{\moplaxm{\A\loose\op}{\CAT}}{W\loose}{\HomC{\S\loose}{H\loose(-)}{Q}},
	\end{equation}
	where the right hand side denotes the marked oplax transformations, that restricts to a $2$-natural isomorphism between the tight parts. Such tight parts are respectively $\HomC{\S\tight}{C}{Q}$ and those marked oplax natural transformations $\alpha\loose$ that restrict to $2$-natural ones $\alpha\tight\:W\tight\aR{}\HomC{\S\tight}{H\tight(-)}{Q}$, i.e.\ those forming a commutative square
	\begin{cd}[6][6]
		W\tight(A) \arrow[r,"{{(\alpha\tight)}_A}"]\arrow[d,"{{\left(j_W\right)}_A}"']\& \HomC{\S\tight}{H\tight(A)}{Q}\arrow[d,"{{\left(J_\S\ast H\tight\right)}_A}"]\\
		W\loose(A) \arrow[r,"{{(\alpha\loose)}_A}"']\& \HomC{\S\loose}{H\loose(A)}{Q}
	\end{cd}
	for every $A\in \A$ (where $J_\S\:\HomC{\S\tight}{-}{Q}\aR{}\HomC{\S\loose}{-}{Q}\c J_\S$)
	
	Remember that identities are always tight and tight morphisms are closed under composition. So the request that the $2$-natural isomorphism of equation~\refs{eqstrictoplaxFcolim} restricts to one between the tight parts equivalently means that the universal marked oplax $2$-cocylinder $\mu^\lambda$ (corresponding to $\id{C}$) satisfies the following two conditions. That for every $A\in \A$ and $X\in W\tight(A)$, the morphism
	$$\mu^\lambda_A(X)\:H(A)\to C$$
	is tight, and that for every $q\:C\to Q$ in $\S$, if $q\c \mu^\lambda_A(X)$ is tight for every $A\in \A$ and $X\in W\tight(A)$ then $q$ needs to be tight. We say that the \dfn{\pteor{cocylinder} $\tau$-components} $\mu^\lambda_A(X)$'s are tight and \dfn{jointly detect tightness}.
\end{rem}
	
\begin{prop}\label{charactsoplaxFcolimit}
	Let $\A$ be a small $\F$-category and consider $\F$-functors $W\:\A\op\to \F$ \pteor{the weight} and $H\:\A\to \S$ \pteor{the $\F$-diagram}. The strict/oplax $\F$-colimit of $H$ weighted by $W$ is, equivalently, an object $C\in \S$ together with a marked oplax $2$-cocylinder
	$$\mu^\lambda\:W\loose\aoplaxm{}\HomC{\S\loose}{H\loose(-)}{C}$$
	that is universal in the $2$-categorical sense, giving a $2$-natural isomorphism
	$$\HomC{\S\loose}{C}{Q}\iso \HomC{\moplaxm{\A\loose\op}{\CAT}}{W\loose}{\HomC{\S\loose}{H\loose(-)}{Q}},$$
	and has $\tau$-components that are tight and jointly detect tightness.
\end{prop}
\begin{proof}
	The proof is clear after \remx\ref{remunivcoconetightjoitnlydetect}. Since the loose limit is $2$-categorical, it can indeed be characterized as having a universal marked oplax $2$-cocylinder.
\end{proof}

\begin{defne}
	We call a strict/oplax $\F$-colimit \dfn{tight} if it is exhibited by a marked oplax $2$-cocylinder $\mu^\lambda$ as in \prox\ref{charactsoplaxFcolimit} such that all \dfn{cocylinder $\lambda$-components} $\mu^\lambda_A(X)$, for $A\in \A$ and $X\in W\loose(A)$, are tight. We will see that this condition is automatic in the case of cartesian-marked oplax $2$-cocones, that is the one we are mostly interested in, as every weighted $2$-cocylinder can be reduced to one of this form.
\end{defne}

\begin{rem}
	We are now ready to prove that having a right-semi-lax right $\F$-adjoint guarantees the preservation of all tight strict/oplax $\F$-colimits. We will actually see that the property of a universal marked oplax $2$-cocylinder to have $\tau$-components that jointly detect tightness is preserved by a right-semi-lax (tight) left $\F$-adjoint, but not necessary to prove the preservation of the rest of the structure, for which a loose adjunction is enough.
	
	The following theorem does not seem to appear in the literature.
\end{rem}

\begin{teor}\label{teorsemilaxFadjpreserve}
	Right-semi-lax loose left $\F$-adjoints preserve all the universal marked oplax $2$-cocylinders for an $\F$-diagram which have tight $\lambda$-components \pteor{i.e.\ in some sense the ``loose part" of all the tight strict/oplax $\F$-colimits, even if the $\tau$-components do not jointly detect tightness}.
	
	Right-semi-lax \pteor{tight} left $\F$-adjoints preserve all tight strict/oplax $\F$-colimits.
\end{teor}
\begin{proof}
	Let $(F,U,\eta,\eps,s,t)$ be a right-semi-lax loose $\F$-adjunction between $\F$-categories $\S$ and $\E$. That is, a lax adjunction between the loose parts where $F$ and $U$ are $\F$-functors, $\eta$ is a loose strict/lax $\F$-natural transformation, $\eps$ is strictly $2$-natural and $s$ is the identity. Let then $\A$ be a small $\F$-category and consider $\F$-functors $W\:\A\op\to \F$ and $H\:\A\to \S$ such that the strict/oplax $\F$-colimit of $H$ weighted by $W$ exists in $\S$ and is tight. Call $C$ such colimit; we want to show that $F$ preserves it. By \prox\ref{charactsoplaxFcolimit}, it suffices to consider the universal marked oplax $2$-cocylinder
	$$\mu^\lambda\:W\loose\aoplaxm{}\HomC{\S\loose}{H\loose(-)}{C}$$
	with tight $\lambda$-components that exhibits $C=\oplaxFcolim{W}{H}$ and the marked oplax $2$-cocylinder $F\c \mu^\lambda$.
	
	We prove that $F\c \mu^\lambda$ is universal in the $2$-categorical sense and such that the $F(\mu^\lambda_A(X))$'s, for $A\in \A$ and $X\in W\loose(A)$, are all tight, without using that the $\tau$-components jointly detect tightness. Moreover, we show that if $\eta$ and $\eps$ have tight components, giving a right-semi-lax (tight) $\F$-adjunction, then having $\tau$-components that jointly detect tightness is preserved as well.
	
	Since a right-semi-lax loose $\F$-adjunction is in particular a right-semi-lax adjunction between the loose parts, we know by \recx\ref{reclaxadjunction} that $(F,U,\eta,\eps,\id{},t)$ induces an adjunction between homsets
	\begin{cd}[6][6]
		\HomC{\E\loose}{F\loose(Y)}{Z}\arrow[r,bend left,"{S_{Y,Z}}",""{name=A}]\& \HomC{\S\loose}{Y}{U\loose(Z)}\arrow[l,bend left,"{T_{Y,Z}}",""{name=B}]
		\Arb{adjunction}{}{A}{B}{0}
	\end{cd}
	for every $Y\in \S$ and $Z\in \E$ with unit the identity, showing $T\c S=\Id{}$, and counit $\xi\:S\c T\aR{}\id{}$. The strategy will be to make use of the equality $T\c S=\Id{}$ to move back and forth between $\E$ and $\S$, recovering, after $T\c S$, the original starting data of $\E$ but with new information gathered in $\S$.
	
	The $\lambda$-components $F(\mu^\lambda_A(X))$'s are surely tight, since $F$ is an $\F$-functor. Take then $q\:F(C)\to Z$ in $\E$ such that $q\c F(\mu^\lambda_A(X))$ is tight for every $A\in \A$ and every $X\in W\tight(A)$. We show that if $\eta$ and $\eps$ have tight components, then $q$ needs to be tight as well. Notice that $S_{-,Z}$ is marked oplax natural in $Y\in \S\op\loose$, with structure $2$-cell on $y\:Y\al{} Y'$ in $\S\loose$  given by $(-\ast \eta_{y})\ast U$, since $\eta$ is loose strict/lax $\F$-natural. Moreover, if $\eta$ has tight components then $S_{-,Z}$ is tight as well, since $U$ is an $\F$-functor. Since $\mu^\lambda_A(X)\:H(A)\to C$ is tight, we obtain $S_{\mu^\lambda_A(X),Z}=\id{}$ and hence
	$$S_{H(A),Z}(q\c F(\mu^\lambda_A(X)))=S_{C,Z}(q)\c \mu^\lambda_A(X)$$
	So if $\eta$ is tight then the left hand side of the equality here above is tight, and since the $\mu^\lambda_A(X)$'s jointly detect tightness we obtain that $S_{C,Z}(q)$ is tight. If we also assume that $\eps$ is tight, then $T_{C,Z}$ is tight, whence $q=T(S(q))$ is tight.
	
	We now prove that $F\c \mu^\lambda$ is universal, assuming only a right-semi-lax loose $\F$-adjunction and never using that the $\tau$-components $\mu^\lambda_A(X)$ jointly detect tightness. Everything below will be loose, so we abuse the notation dropping the loose subscripts. The following figure condenses the strategy.
	\begin{eqD}{diagrcondensesprescolim}
	\begin{cd}*[3.5][3.5]
		W\arrow[d,Rightarrow,"{\sigma}"']\arrow[r,Rightarrow,"{\mu}"] \& \HomC{\S}{H(-)}{C}\arrow[d,Rightarrow,dashed,"{\gamma\c -}"] \arrow[r,Rightarrow,"{F}"]\& \HomC{\E}{(F\c H)(-)}{F(C)} \arrow[d,dashed,Rightarrow,"{T(\gamma)\c -}"]\\
		\HomC{\E}{(F\c H)(-)}{Z} \arrow[r,Rightarrow,"{S}"']\& \HomC{\S}{H(-)}{U(Z)} \arrow[r,Rightarrow,"{T}"']\& \HomC{\E}{(F\c H)(-)}{Z}
	\end{cd}
	\end{eqD}
	Given a marked oplax $2$-cocylinder
	$$\sigma\:W\aoplaxm{}\HomC{\E}{(F\c H)(-)}{Z},$$
	we prove that there is a unique $\delta\:F(C)\to Z$ in $\E$ such that $(\delta\c -)\c F\c \mu = \sigma$.
	Postcomposing $\sigma$ with $S_{H(-),Z}$, we obtain a marked oplax $2$-cocylinder for $H$. Indeed $S_{H(-),Z}$ is marked oplax natural in $A\in \A\op$ with structure $2$-cell on $a\:A\al{}A'$ in $\A$ given by $(-\ast \eta_{H(a)})\ast U$, since $\eta$ is loose strict/lax $\F$-natural and $H$ is an $\F$-functor. So, by universality of $\mu$, $S\c \sigma$ induces a unique $\gamma\:C\to U(Z)$ in $\S$ such that $(\gamma\c -)\c \mu=S\c \sigma$.
	Notice then that the right square of the figure in equation~\refs{diagrcondensesprescolim} is commutative, since it is equivalent to
	$$(\eps_{Z}\c F(\gamma))\c F(-)=\eps_Z\c F(\gamma\c -),$$
	which holds since $F$ is a $2$-functor. So $\delta\deq T(\gamma)$ in $\E$ is such that
	$$(\delta\c -)\c F\c \mu=T\c (\gamma\c -)\c \mu=T\c S\c \sigma=\sigma.$$
	We now show the uniqueness of $\delta$. So consider another $\delta'\:F(C)\to Z$ in $\E$ such that $(\delta'\c -)\c F\c \mu = \sigma$. Postcomposing with $S$ we obtain
	$$S\c (\delta'\c -)\c F\c \mu = S\c \sigma.$$
	But we notice that
	\begin{eqD}{eqFcattheorem}
		S\c (\delta'\c -)\c F\c \mu=\left(S(\delta')\c -\right)\c \mu
	\end{eqD}
	as marked oplax natural transformations. Indeed, given $A\in \A$ and $\alpha\:X\to X'$ in $W(A)$,
	$${\left(S\c (\delta'\c -)\c F\c \mu\right)}_{A}(X)=U\left(\delta'\c F\left(\mu_A(X)\right)\right)\c \eta_{H(A)}=U(\delta')\c U\left(F\left(\mu_A(X)\right)\right)\c \eta_{H(A)}.$$
	Since $\mu_A(X)$ is tight and $\eta$ is loose strict/lax $\F$-natural, $\eta_{\mu_A(X)}=\id{}$ and hence
	$$U(\delta')\c U\left(F\left(\mu_A(X)\right)\right)\c \eta_{H(A)}=U(\delta')\c \eta_C\c \mu_A(X)=S(\delta')\c \mu_A(X).$$
	And it works similarly for the images on $\alpha$, using the $2$-dimensional property of $\eta$ being oplax natural. Given $a\:A\al{} A'$ in $\A$ and $X\in W(A)$,
	$${\left(S\c (\delta'\c -)\c F\c \mu\right)}_{a,X}=U\left(\delta'\c F\left(\mu_A(X)\right)\right)\ast \eta_{H(a)} \c U\left(\delta'\ast F\left(\mu_{a,X}\right)\right)\ast \eta_{H(A')}.$$
	Considering $\mu_{a,X}\:\mu_{A'}(W(a)(X))\aR{} \mu_A(X)\c H(a)$ in $\S$, since $\eta$ is loose strict/lax $\F$-natural and both $\mu_A(X)$ and $\mu_{A'}(W(a)(X))$ are tight, we obtain
	$$U\left(F\left(\mu_A(X)\right)\right)\ast\eta_{H(a)}\c U\left(F\left(\mu_{a,X}\right)\right)\ast \eta_{H(A')}=\eta_C\ast \mu_{a,X},$$
	whence we conclude that equation~\refs{eqFcattheorem} holds. Therefore
	$$\left(S(\delta')\c -\right)\c \mu=S\c \sigma=(\gamma\c -)\c \mu,$$
	and by universality of $\mu$ we conclude that $S(\delta')=\gamma$, whence 
	$$\delta'=T(S(\delta'))= T(\gamma)=\delta.$$
	
	It remains to prove the $2$-dimensional universal property of $F\c \mu$. Given a modification
	$$\Sigma\:\sigma\aM{}\sigma'\:W\aoplaxm{}\HomC{\E}{(F\c H)(-)}{Z},$$
	we show that there is a unique $\Delta\:\delta\aR{}\delta'\:F(C)\to Z$ in $\E$ such that $(\Delta\ast -)\ast (F\c \mu)=\Sigma$.
	By universality of $\mu$, whiskering $\Sigma$ with $S$ on the right induces a unique $\Gamma\:\gamma\aR{}\gamma'$ in $\S$ such that $(\Gamma\ast -)\ast \mu = S\ast \Sigma$. Notice then that
	$$T\ast (\Gamma\ast -)=(T(\Gamma)\ast -)\ast F$$
	because $F$ is a $2$-functor, and thus $\Delta\deq T(\Gamma)$ in $\E$ is such that
	$$(\Delta\ast -)\ast (F\c \mu)=T\ast (\Gamma\ast -)\ast \mu =T\ast S\ast \Sigma=\Sigma.$$
	To show the uniqueness of $\Delta$, take another $\Delta'$ such that $(\Delta'\ast -)\ast (F\c \mu)=\Sigma$. Whiskering with $S$ on the right, we obtain
	$$S\ast (\Delta'\ast -)\ast (F\c \mu)=S\ast \Sigma.$$
	But notice that
	$$S\ast (\Delta'\ast -)\ast (F\c \mu)=(S(\Delta')\ast -)\ast \mu$$
	Indeed it suffices to check it on components, where it holds since $\mu_A(X)$ is tight and hence $\eta_{\mu_A(X)}=\id{}$. So
	$$(S(\Delta')\ast -)\ast \mu=S\ast \Sigma=(\Gamma\ast -)\ast \mu,$$
	whence $S(\Delta')=\Gamma$ by universality of $\mu$ and thus
	$$\Delta'=T(S(\Delta'))=T(\Gamma)=\Delta.$$
	Therefore $F\c\mu$ is universal in the $2$-categorical sense.
\end{proof}

\section{Colimits in 2-dimensional slices}\label{sectioncolimitsin2slices}

We aim at generalizing to dimension $2$ the well-known $1$-dimensional result that a colimit in a slice corresponds to the map from the colimit of the domains of the diagram which is induced by the universal property. Half of such result will be captured by preservation of $2$-colimits for the domain $2$-functor from a lax slice (assuming to have products in the base 2-category), that we will address in Section~\ref{sectionliftingpreservfordom}. In this section, we focus on the other half, that is the generalization to dimension $2$ of the special property of equation~\refs{propertycoliminslicedim1} (of \thex\ref{colimitsinslicesdim1}). Namely, we want to prove that a morphism from a $2$-colimit to some $M$ can be expressed as a $2$-colimit in a $2$-dimensional slice over $M$.

We show two different approaches to this, that lead to the same result (compare \conx\ref{consfirstapproach} with \thex\ref{teordomis2colimfib}). The first one is more intuitive, based on the reduction of the weighted $2$-colimits to cartesian-marked oplax conical ones. Such reduction has been originally proved by Street in~\cite{street_limitsindexedbycatvalued}; we gave new more elementary proofs in~\cite{mesiti_twosetenrgrothconstrlaxnlimits}. See also Descotte, Dubuc and Szyld's~\cite{descottedubucszyld_sigmalimandflatpseudofun} for a pseudo version. The second approach is more abstract, based on an original concept of \predfn{colim-fibration} (\defx\ref{defcolimfibrationdim1} in dimension $1$ and \defx\ref{def2colimfib} in dimension $2$). This will offer a shorter and more elegant proof, in \thex\ref{teordomis2colimfib}. Both the approaches show the need to consider lax slices.

This section also contains a result of reflection of $2$-colimits for the domain $2$-functor $\dom\:\laxslice{\E}{M}\to \E$ from a lax slice.

We now begin exploring the first approach to the generalization to dimension $2$ of equation~\refs{propertycoliminslicedim1} (of \thex\ref{colimitsinslicesdim1}).

\begin{rec}[Street's~\cite{street_limitsindexedbycatvalued}, our~\cite{mesiti_twosetenrgrothconstrlaxnlimits}]\label{rectwocatofel}
	We recall here the 2-category of elements construction and the reduction of weighted 2-colimits to cartesian-marked oplax conical ones.

	Let $W\:\A\op\to \Cat$ be a 2-functor with $\A$ a small 2-category. The \dfn{2-category of elements} of $W$ is the 2-functor $\groth{W}\:\Groth{W}\to\A$  of projection on the first component, with $\Groth{W}$ defined as follows:
		\begin{description}
			\item[an object of $\Groth{W}$] is a pair $(A,X)$ with $A\in \A$ and $X\in F(A)$;
			\item[a morphism $(A,X)\to (B,X')$ in $\Groth{W}$] is a pair $(f,\alpha)$ with $f\:A\to B$ a morphism in $\A$ and $\alpha\:X\to W(f)(X')$ a morphism in $W(A)$;
			\item[a $2$-cell $(f,\alpha)\aR{}(g,\beta)\:(A,X)\to (B,X')$ in $\Groth{W}$] is a $2$-cell $\delta\:f\aR{}g$ in $\A$ such that $W(\delta)_{X'}\c \alpha=\beta$.
		\end{description}

	The 2-category of elements gives a 2-equivalence (\cite{mesiti_twosetenrgrothconstrlaxnlimits}) between the 2-category $\m{\A\op}{\Cat}$ of 2-presheaves and that of \predfn{split discrete 2-fibrations} with small fibres over $\A$.

	A \dfn{cloven \textup{[}resp.\ split\textup{]} discrete 2-fibration} is a $2$-functor $P\:\E\to \A$ such that
		\begin{enum}
			\item the underlying functor $P_0$ of $P$ is an ordinary cloven \textup{[}resp.\ split\textup{]} Grothendieck fibration;
			\item for every pair $X,Y\in \E$ the functor $P_{X,Y}\:\HomC{\E}{X}{Y}\to \HomC{\B}{P(X)}{P(Y)}$ is a discrete opfibration.
		\end{enum}
	By Lambert's~\cite{lambert_discretetwofib}, this corresponds to a locally discrete version of Hermida's~\cite{hermida_somepropoffibasafibredtwocat} $2$-fibrations.

	$\Groth{W}$ is naturally endowed with a structure of $\F$-category, taking its loose part to be itself (as a 2-category) and the tight morphisms to be the the morphisms of type $\p{f,\id{}}$, i.e.\ the morphisms of the cleavage naturally associated to the discrete 2-fibration $\groth{W}\:\Groth{W}\to\A$.

	Consider $2$-functors $M,N\:{\left(\Groth{W}\right)}\op\to \Cat$. A \dfn{cartesian-marked oplax natural transformation} $\alpha$ from $M$ to $N$, denoted $\alpha\:M\aoplaxn{}N$,\v is an oplax natural transformation $\alpha$ from $M$ to $N$ such that the structure $2$-cell on every morphism $\p{f,\id{}}$ in ${\left(\Groth{W}\right)}\op$ is the identity. Equivalently, it is a marked oplax natural transformation with respect to the $\F$-category structure on $\Groth{W}$ described above.

	Let $W\:\A\op\to \Cat$ and $F\:\Groth{W}\to \C$ be 2-functors with $\A$ small. The \dfn{cartesian-marked oplax conical $2$-colimit of $F$}, denoted as $\oplaxncolim{F}$, is (if it exists) an object $C\in \C$ together with a $2$-natural isomorphism of categories
		$$\HomC{\C}{C}{U}\iso \HomC{\moplaxn{{\left(\Grothdiag{W}\right)}\op}{\CAT}}{\Delta 1}{\HomC{\C}{F(-)}{U}}$$
	where the right hand side denotes cartesian-marked oplax natural transformations. Equivalently, this is the loose part of a strict/oplax $\F$-colimit with respect to the $\F$-category structure on $\Groth{W}$ described above. Taking the trivial $\F$-category structure on $\C$, with the tight morphisms being all loose morphisms, and considering the conical $\F$-weight $\Delta 1$, this is automatically a tight strict/oplax $\F$-colimit.
		
		\noindent When $\oplaxncolim{F}$ exists, the identity on $C$ provides the \dfn{universal cartesian-marked oplax cocone} $\mu\:\Delta 1 \aoplaxn{}\HomC{\C}{F(-)}{C}$.

	 Cartesian-marked oplax conical $2$-colimits are particular weighted $2$-colimits.

	Every weighted $2$-colimit can be reduced to a cartesian-marked oplax conical one. Given $2$-functors $F\:\A\to \C$ and $W\:\A\op\to\CAT$ with $\A$ small,
	$$\wcolim{W}{F}\iso \oplaxncolim{\left(F\c \groth{W}\right)}.$$
\end{rec}

\begin{cons}\label{consfirstapproach}
	Let $\E$ be a $2$-category and let $M\in \E$. Consider a $2$-diagram $F\:\A\to \E$ with $\A$ small and a weight $W\:\A\op\to \CAT$ such that the colimit $\wcolim{W}{F}$ of $F$ weighted by $W$ exists in $\E$. Take then a morphism $q\:\wcolim{W}{F}\to M$, or equivalently the corresponding weighted $2$-cocylinder
	$$\nu^q\:W\aR{}\HomC{\E}{F(-)}{M}.$$
	We would like to express $q$ as a $2$-colimit in a $2$-dimensional slice of $\E$ over $M$. So we need to construct from $\nu^q$ a $2$-diagram in a $2$-dimensional slice. In dimension $1$, equation~\refs{propertycoliminslicedim1} (of \thex\ref{colimitsinslicesdim1}) is based on the fact that a cocone on $M$ coincides with a diagram in $\slice{\C}{M}$. But in dimension $2$ we have a weighted $2$-cocylinder $\nu^q$ instead of a strict cocone, and thus it is not clear how to directly find a corresponding diagram in a slice, which still has a conical shape. We notice that this is essentially a matter of selecting a cocone out of the bunch of cocones that form the weighted $2$-cocylinder $\nu^q$. We then obtain a great help from the reduction of weighted $2$-colimits to cartesian-marked oplax conical ones.
	$$\wcolim{W}{F}\iso \oplaxncolim{\left(F\c \groth{W}\right)}$$
	where $\groth{W}\:\Groth{W}\to\A$ is the $2$-category of elements of $W$. And $\nu^q$ corresponds to a cartesian-marked oplax $2$-cocone
	$$\lambda^q\:\Delta 1 \aoplaxn{} \HomC{\E}{(F\c \groth{W})(-)}{M}\:{\left(\Grothdiag{W}\right)}\op\to \CAT.$$
	
	It is now easy to check that a cartesian-marked oplax $2$-cocone on $M$ can be reorganized as a $2$-diagram in the lax slice $\laxslice{\E}{M}$ on $M$ (we will see the complete correspondence in \prox\ref{correspoplaxncoconesdiag}), where a $1$-cell in the lax slice from $E\ar{g}{M}$ to $E'\ar{g'}{M}$ is a filled triangle
	\tr[4][5.5][0][0][l][-0.3][\gamma][0.85]{E}{E'}{M}{\widehat{\gamma}}{g}{g'}
	More precisely, we can reorganize $\lambda^q$ as the $2$-diagram
	\begin{fun}
		L^q & \: & \Grothdiag{W}\hphantom{...} & \too & \hphantom{...} \laxslice{\E}{M} \\[1.3ex]
		&& \fibdiag{(A,X)}{(f,\alpha)}{(B,X')} & \mto & \tr*[2.8][4][0][0][l][-0.235][\lambda^q_{f,\alpha}][1]{F(A)}{F(B)}{M}{F(f)}{\lambda^q_{(A,X)}}{\lambda^q_{(B,X')}}\\[1.3ex]
		&& \delta \hphantom{.....}& \mto & \hphantom{.....} F(\delta)
	\end{fun}
	In \thex\ref{teordomis2colimfib}, we will prove that
	$$\fib{\wcolim{W}{F}}{q}{M}=\fib{\oplaxncolim{\left(F\c \groth{W}\right)}}{q}{M}=\oplaxncolim{L^q}$$
	in the lax slice $\laxslice{\E}{M}$. Of course, one could prove this directly, but our proof will be shorter and more abstract, in \thex\ref{teordomis2colimfib}, based on the \predfn{colim-fibrations} point of view (that is, the second approach named above). 
\end{cons}

\begin{rem}
	Although weighted $2$-colimits cannot be conicalized, we can almost conicalize them reducing them to cartesian-marked oplax conical $2$-colimits. The price to pay is to have $2$-cells inside the cocones. And this then translates as the need to consider lax slices in order to generalize the $1$-dimensional \thex\ref{colimitsinslicesdim1} to dimension $2$.

	Such need is further justified by the second approach (see \remx\ref{remjustifydomlaxslice}), that we now present. The idea is to capture \thex\ref{colimitsinslicesdim1} from a more abstract point of view, in a way that resembles the property of being a discrete fibration. We will then proceed to generalize such approach to dimension $2$, arriving to \thex\ref{teordomis2colimfib}.
	
	The following definition does not seem to appear in the literature. 
\end{rem}

\begin{defne}\label{defcolimfibrationdim1}
	A functor $p\:\S\to \C$ is a \dfn{colim-fibration} if for every object $S\in \S$ and every universal cocone $\mu$ that exhibits $p(S)$ as the colimit of some diagram $D\:\A\to \C$ with $\A$ small, there exists a unique pair $(\ov{D},\ov{\mu})$ with $\ov{D}\:\A\to\S$ a diagram and $\ov{\mu}$ a universal cocone that exhibits $S=\colim{\ov{D}}$ such that $p\c \ov{D}=D$ and $p\c \ov{\mu}=\mu$.\nv
	\begin{eqD}{imagecolimfibration}
		\begin{cd}*[1.6][0.5]
			\ov{D}(A)\arrow[rrd,bend left=12,"{\ov{\mu}_A}"]\arrow[dr,"{\ov{D}(f)}"',""{name=A}]\&[-0.2ex]\\
			\& \ov{D}(B)\arrow[dd,mapsto,"{p}"{pos=0.164},shorten <=-0.2ex,shorten >=6.7ex,shift right=4.35ex]\arrow[r,"{\ov{\mu}_B}"']\&[10ex] S\arrow[dd,mapsto,"{p}"{pos=0.51},shorten <=3.25ex,shorten >=3.1ex]\\[1ex]
			{D}(A)\arrow[rrd,bend left=12,"{{\mu}_A}"]\arrow[dr,"{{D}(f)}"',""{name=B}]\\
			\& {D}(B)\arrow[r,"{{\mu}_B}"']\&[10ex] p(S)
		\end{cd}
	\end{eqD}
\end{defne}

\begin{prop}\label{characterizecolimfibration}
	Let $p\:\S\to\C$ be a functor. The following are equivalent:
	\begin{enumT}
		\item $p$ is a colim-fibration;
		\item $p$ is a discrete fibration that reflects all colimits;
		\item for every object $S\in \S$ and every universal cocone $\mu$ that exhibits $p(S)$ as the colimit of some diagram $D\:\A\to \C$ with $\A$ small, there exists a unique pair $(\ov{D},\ov{\mu})$ with $\ov{D}\:\A\to\S$ a diagram and a $\ov{\mu}$ a cocone for $\ov{D}$ on $S$ such that $p\c \ov{D}=D$ and $p\c \ov{\mu}=\mu$; moreover $\ov{\mu}$ is a universal cocone that exhibits $S=\colim{\ov{D}}$.
	\end{enumT}
\end{prop}
\begin{proof}
	It is straightforward to prove that any colim-fibration is a discrete fibration, considering diagrams $D\:\2\to \C$ that correspond to the arrows to lift. As a corollary, we obtain that $(i)$ is equivalent to $(iii)$.
	
	We show $``(i)\aR{}(ii)"$. Take a diagram $H\:\A\to \S$ with $\A$ small and a cocone $\zeta$ for $H$ on some object $S\in\S$; assume then that $p\c \zeta$ is a universal cocone, exhibiting $p(S)=\colim{(p\c H)}$.  By condition $(iii)$, we know that there exist a unique pair $(\ov{p\c H},\ov{p\c \zeta})$ with $\ov{p\c H}\:\A\to\S$ a diagram and $\ov{p\c \zeta}$ a cocone for $\ov{p\c H}$ on $S$ such that $p\c\ov{p\c H}=p\c H$ and $p\c \ov{p\c \zeta} = p\c \zeta$, and that moreover $\ov{p\c \zeta}$ is a universal cocone that exhibits $S=\colim{\ov{p\c H}}$. But then we need to have that $\ov{p\c H}=H$ and $\ov{p\c \zeta}=\zeta$.
	
	Finally, we prove $``(ii)\aR{}(i)"$. Take $S\in \S$ and a universal cocone $\mu$ that exhibits $p(S)$ as the colimit of some diagram $D\:\A\to \C$ with $\A$ small. Since $p$ is a discrete fibration, there exists a unique pair $(\ov{D},\ov{\mu})$ with $\ov{D}\:\A\to\S$ a diagram and $\ov{\mu}$ a cocone for $\ov{D}$ on $S$ such that $p\c \ov{D}=D$ and $p\c \ov{\mu}=\mu$. Indeed, for every $A\in \A$, we can define $\ov{D}(A)$ and $\ov{\mu}_A$ by taking the unique lifting of $\mu_A$ to $S$, and, for every $f\:A\to B$ in $\A$, define $\ov{D}(f)$ to be the unique lifting of $D(f)$ to $\ov{D}(B)$, whose domain needs to be $\ov{D}(A)$ since any discrete fibration is split. Moreover $\ov{\mu}$ needs to be universal since $p$ reflects all the colimits and $p\c \ov{\mu}=\mu$ is universal.
\end{proof}
	
\begin{rem}
	Condition $(iii)$ of \prox\ref{characterizecolimfibration} says that, for a colim-fibration, the liftings $\ov{\mu}$ of universal cocones $\mu$ are unique as mere cocones over $\mu$ on the starting $S\in \S$.
	
	We notice that such condition $(iii)$ is actually similar to the definition of creating colimits (see for example Ad{\'a}mek, Herrlich and Strecker's~\cite{adamekherrlichstrecker_abstractandconcretecat}), but somehow dual to it. Indeed, looking at the diagram in equation~\refs{imagecolimfibration}, creation of colimits starts from a diagram $\ov{D}$ and produces a colimit $S$ for it, while being a colim-fibration starts from some $S$ and produces a diagram $\ov{D}$ with colimit $S$.
\end{rem}
	
\begin{rem}\label{teordomrewritten}
	We can now rewrite \thex\ref{colimitsinslicesdim1} by saying that $\dom\:\slice{\C}{M}\to\C$ is a colim-fibration that preserves and uniquely lifts (or actually just preserves and detects) all colimits (assuming products in $\C$). Indeed, this is stronger than \thex\ref{colimitsinslicesdim1}, but $\dom$ can be expressed as the category of elements of the representable $\y{M}\:\C\op\to\Set$ and is thus a discrete fibration. The explicit formula in equation~\refs{propertycoliminslicedim1} comes from the explicit liftings of $\dom$.
\end{rem}

\begin{cons}\label{consliftoplaxnormalcocone}
	At this point, we are ready to generalize the concept of colim-fibration to dimension $2$, upgrading discrete fibrations to discrete 2-fibrations. As discrete 2-fibrations are locally discrete opfibrations, they are able to uniquely lift $2$-cells to a fixed domain $1$-cell. As it would now be much harder to directly generalize \defx\ref{defcolimfibrationdim1} in a way that implies being a discrete $2$-fibration, we think it is most concise to just ask having a discrete $2$-fibration.

	We then need to use a $2$-categorical concept of cocone. While weighted $2$-cocylinders would be hard to handle, we notice that a discrete $2$-fibration has the ability to lift cartesian-marked oplax $2$-cocones. Indeed, let $p\:\S\to\E$ be a cloven discrete $2$-fibration. Consider then $S\in \S$, a marking $W\:\A\op\to \CAT$ with $\A$ small, a $2$-diagram $D\:\Groth{W}\to \E$ and a cartesian-marked oplax $2$-cocone
	$$\theta\:\Delta 1\aoplaxn{}\HomC{\E}{D(-)}{p(S)}$$
	for $D$ on $p(S)$. Then $p$ lifts $(D,\theta)$ to a pair $(\ov{D},\ov{\theta})$ with $\ov{D}\:\Groth{W}\to \S$ a $2$-diagram and $\ov{\theta}$ a cartesian-marked oplax $2$-cocone for $\ov{D}$ on $S$ such that $p\c \ov{D}=D$ and $p\c \ov{\theta}=\theta$.
	\begin{eqD}{imagecolimtwofibration}
		\begin{cd}*[1.5][0.6]
			\ov{D}(A,X)\arrow[rrd,bend left=20,"{\ov{\theta}_{(A,X)}}"{inner sep=0.2ex},""'{name=J}]\arrow[dr,"{\ov{D}(f,\alpha)}"',""{name=A}]\&[-0.6ex]\\
			\& \ov{D}(B,X')\arrow[Rightarrow,from=J,"{\ov{\theta}_{f,\alpha}}"{pos=0.34},shorten <=0.75ex,shorten >=0.65ex]\arrow[dd,mapsto,"{p}"{pos=0.20},shorten <=0.15ex,shorten >=7.15ex,shift right=5.7ex]\arrow[r,"{\ov{\theta}_{(B,X')}}"']\&[10ex] S\arrow[dd,mapsto,"{p}"{pos=0.51},shorten <=3.75ex,shorten >=3.4ex]\\[3ex]
			{D}(A,X)\arrow[rrd,bend left=20,"{{\theta}_{(A,X)}}"{inner sep=0.2ex},""'{name=K}]\arrow[dr,"{{D}(f,\alpha)}"',""{name=B}]\\
			\& {D}(B,X')\arrow[Rightarrow,from=K,"{\theta_{f,\alpha}}"{pos=0.34},shorten <=0.75ex,shorten >=0.65ex]\arrow[r,"{{\theta}_{(B,X')}}"']\&[10ex] p(S)
		\end{cd}
	\end{eqD}
	For every $(A,X)\in \Groth{W}$, we define $\ov{D}(A,X)$ and $\ov{\theta}_{(A,X)}$ by taking the chosen cartesian lifting of $\theta_{(A,X)}$ to $S$. For every $(f,\alpha)\:(A,X)\to (B,X')$ in $\Groth{W}$, we then define $\ov{\theta}_{f,\alpha}\:\ov{\theta}_{(A,X)}\to \xi$ to be the unique lifting of $\theta_{f,\alpha}$ to $\ov{\theta}_{(A,X)}$. Since $\ov{\theta}_{(B,X')}$ is cartesian, $\xi$ factors through $\ov{\theta}_{(B,X')}$; we define $\ov{D}(f,\alpha)$ to be the unique factoring morphism $\ov{D}(A,X)\to\ov{D}(B,X')$, so that $\ov{\theta}_{f,\alpha}\:\ov{\theta}_{(A,X)}\to \ov{\theta}_{(B,X')}\c \ov{D}(f,\alpha)$. It remains to define $\ov{D}$ on $2$-cells. Given $\delta\:(f,\alpha)\aR{}(g,\beta)\:(A,X)\to (B,X')$ in $\Groth{W}$, we define $\ov{D}(\delta)$ to be the unique lifting of $D(\delta)$ to $\ov{D}(f,\alpha)$. The codomain of $\ov{D}(\delta)$ is $\ov{D}(g,\beta)$ because of the uniqueness of the lifting of 
	\begin{cd}[2.4][1]
		{D}(A,X)\arrow[rrd,bend left=20,"{{\theta}_{(A,X)}}"{inner sep=0.2ex},""'{name=K}]\arrow[dr,bend left=18,"{{D}(f,\alpha)}"{inner sep=0.15ex},""'{name=B}]\arrow[dr,bend right=38,"{{D}(g,\beta)}"'{inner sep=0.15ex},""{name=C}]\\
		\&{D}(B,X')\arrow[Rightarrow,from=K,"{\theta_{f,\alpha}}"{pos=0.34},shorten <=0.75ex,shorten >=0.65ex]\arrow[r,"{{\theta}_{(B,X')}}"']\&[10ex] p(S),
		\Arbs[pos=0.2]{Rightarrow,shift left=0.9ex}{D(\delta)}{B}{C}{0.1}
	\end{cd}
	which coincides with $\theta_{g,\beta}$, to $\ov{\theta}_{(A,X)}$, and the cartesianity of $\ov{\theta}_{(B,X')}$. This argument also proves the $2$-dimensional property of the oplax naturality of $\ov{\theta}$. It is straightforward to check that $\ov{D}$ is a $2$-functor and that $\ov{\theta}$ is cartesian-marked oplax natural, using the cartesianity of the $\ov{\theta}_{(A,X)}$'s and the uniqueness of the liftings of a $2$-cell to a fixed domain $1$-cell (with arguments similar to the above one).

	Clearly, the $\ov{\theta}_{(A,X)}$'s are not unique above the $\theta_{(A,X)}$, but cartesian. Having fixed them, however, the rest of the cartesian-marked oplax cocone $\ov{\theta}$ is uniquely defined. It is also true that, given another pair $(\t{D},\t{\theta})$ that lifts $(D,\theta)$, the unique vertical morphisms $\t{D}(A,X)\to \ov{D}(A,X)$ that produce the factorization of the $\t{\theta}_{(A,X)}$'s through the cartesian $\ov{\theta}_{(A,X)}$'s form a unique vertical $2$-natural transformation $j\:\t{D}\to \ov{D}$ such that $\t{\theta}=(-\c j)\c\ov{\theta}$. This can be checked using the uniqueness of the liftings of a $2$-cell to a fixed domain $1$-cell and the cartesianity of the $\ov{\theta}_{(A,X)}$'s.
\end{cons}

The following definition is original.

\begin{defne}\label{def2colimfib}
	A $2$-functor $p\:\S\to \E$ is a \dfn{$2$-colim-fibration} if it is a cloven discrete $2$-fibration such that, for every $S\in \S$, marking $W\:\A\op\to \CAT$ with $\A$ small, $2$-diagram $D\:\Groth{W}\to \E$ and universal cartesian-marked oplax $2$-cocone
	$$\theta\:\Delta 1\aoplaxn{}\HomC{\E}{D(-)}{p(S)}$$
	that exhibits $p(S)=\oplaxncolim{D}$, the pair $(\ov{D},\ov{\theta})$ obtained by lifting $(D,\theta)$ through $p$ to $S$ as in \conx\ref{consliftoplaxnormalcocone} exhibits
	$$S=\oplaxncolim{\ov{D}}.$$
\end{defne}

\begin{rem}
	Remember that every weighted $2$-colimit can be reduced to a cartesian-marked oplax conical one, so the property of being a $2$-colim-fibration can as well be applied to any universal weighted $2$-cocylinder.

	We would now like to generalize \prox\ref{characterizecolimfibration} to dimension $2$. We see, however, that a $2$-colim-fibration does not necessarily reflect all the (cartesian-marked oplax conical) $2$-colimits, because the lifting $(\ov{D},\ov{\theta})$ of \conx\ref{consliftoplaxnormalcocone} is not unique anymore. Indeed, if we start from a cartesian-marked oplax $2$-cocone above (at the level of $\S$), project it down and lift, we do not find in general the starting cartesian-marked oplax $2$-cocone. We almost find the same, however, if we start from a cartesian-marked oplax $2$-cocone that we call \predfn{cartesian}, and now define. This will bring to \prox\ref{propcharacterization2colimfib}.
\end{rem}

\begin{defne}
	Let $p\:\S\to \E$ be a cloven discrete $2$-fibration. Consider then $S\in \S$, a marking $W\:\A\op\to \CAT$ with $\A$ small and a $2$-diagram $H\:\Groth{W}\to \S$. A cartesian-marked oplax $2$-cocone
	$$\zeta\:\Delta 1\aoplaxn{}\HomC{\S}{H(-)}{S}$$
	is \dfn{cartesian} \textup{[}resp.\ \dfn{cloven}\textup{]} if for every $(A,X)\in \Groth{W}$ the component $\zeta_{(A,X)}$ (seen as a morphism in $\S$) is cartesian with respect to $p$ \textup{[}resp.\ in the cleavage\textup{]}.

	We say that $p$ \dfn{reflects all the cartesian} \textup{[}resp.\ \dfn{cloven}\textup{]} (cartesian-marked oplax conical) \dfn{$2$-colimits} if it reflects the universality of cartesian \textup{[}resp.\ cloven\textup{]} cartesian-marked oplax $2$-cocones.
\end{defne}

\begin{prop}\label{propcharacterization2colimfib}
	Let $p\:\S\to \E$ be a cloven discrete $2$-fibration. The following are equivalent:
	\begin{enumT}
		\item $p$ is a $2$-colim-fibration;
		\item $p$ reflects all the cartesian $2$-colimits;
		\item $p$ reflects all the cloven $2$-colimits.
	\end{enumT}
\end{prop}
\begin{proof}
	$``(ii)\aR{}(iii)"$ is trivial. We show $``(iii)\aR{}(i)"$. In the notation of \defx\ref{def2colimfib}, the cartesian-marked oplax $2$-cocone $\ov{\theta}$ is cloven by \conx\ref{consliftoplaxnormalcocone}. Since $p\c \ov{\theta}=\theta$ is universal, also $\ov{\theta}$ needs to be universal.

	We now prove $``(i)\aR{}(ii)"$. So consider $S\in \S$, a marking $W\:\A\op\to \CAT$ with $\A$ small, a $2$-diagram $H\:\Groth{W}\to \S$ and a cartesian cartesian-marked oplax $2$-cocone
	$$\zeta\:\Delta 1\aoplaxn{}\HomC{\S}{H(-)}{S}.$$
	Assume that $p\c \zeta$ is universal, exhibiting $p(S)=\oplaxncolim{(p\c H)}$. We prove that $\zeta$ is universal as well. Consider the lifting $(\ov{p\c H},\ov{p\c \zeta})$ of $({p\c H},{p\c \zeta})$ through $p$ to $S$, as in \conx\ref{consliftoplaxnormalcocone}. It is straightforward to check that, since cartesian liftings are unique up to a unique vertical isomorphism, there exists a $2$-natural isomorphism $j\:H\iso \ov{p\c H}$ such that $\zeta=(-\c j)\c\ov{p\c \zeta}$ (see the last part of \conx\ref{consliftoplaxnormalcocone}). Since $\ov{p\c \zeta}$ is universal, as $p$ is a $2$-colim-fibration, and $(-\c j)$ is a $2$-natural isomorphism, we conclude.
\end{proof}

\begin{rem}\label{remjustifydomlaxslice}\label{excartinlaxslice}
	Looking at \remx\ref{teordomrewritten}, in order to obtain a generalization to dimension 2 of \thex\ref{colimitsinslicesdim1} (or better, of the stronger colim-fibration result), we upgrade the $1$-dimensional $\dom\:\slice{\C}{M}\to\C$ to the 2-category of elements of the representable $\y{M}\:\E\op\to \Cat$. This produces the domain 2-functor from the lax slice $\laxslice{\E}{M}$, further justifying \conx\ref{consfirstapproach}.
	\sq[o][6][6.5][\lax \opn{comma}][2.2][2.7][0.57]{\laxslice{\E}{M}}{\1}{\E}{\CAT\op}{}{\dom}{\1}{\y{M}}
	The 2-functor $\dom\:\laxslice{\E}{M}\to \E$ is thus a cloven discrete 2-fibration (in a canonical way). 

	The cartesian [resp.\ cloven] morphisms in $\laxslice{\E}{M}$ with respect to $\dom$ are precisely the triangles
	\tr[3.6][5][0][0][i][-0.6][\gamma][0.85]{E}{E'}{M}{\widehat{\gamma}}{g}{g'}
	with the $2$-cell $\gamma$ an isomorphism [resp.\ an identity]. So the cartesian [resp.\ cloven] cartesian-marked oplax $2$-cocones in $\laxslice{\E}{M}$ are the ones with components triangles filled with isomorphisms [resp.\ identities].
\end{rem}

\begin{teor}\label{teordomreflects}
	Let $\E$ be a $2$-category and $M\in \E$. Then $\dom\:\laxslice{\E}{M}\to \E$ reflects all the cartesian 2-colimits.
\end{teor}
\begin{proof}
	Take $t\:K\to M$, a marking $W\:\A\op\to \CAT$ with $\A$ small and a $2$-diagram $H\:\Groth{W}\to \laxslice{\E}{M}$. Consider then a cartesian cartesian-marked oplax $2$-cocone
	$$\zeta\:\Delta 1\aoplaxn{}\HomC{\laxslice{\E}{M}}{H(-)}{t}$$
	such that $\dom\c\h \zeta$ is universal. We prove that $\zeta$ is universal as well. Given $g\:E\to M$ and
	$$\sigma\:\Delta 1\aoplaxn{}\HomC{\laxslice{\E}{M}}{H(-)}{g},$$
	we need to produce a morphism
	\tr[3.6][5.1][0][0][l][-0.3][\gamma][0.85]{K}{E}{M}{\widehat{\gamma}}{t}{g}
	in $\laxslice{\E}{M}$ such that $\sigma=(\gamma\c -)\c \zeta$. We then need
	$$\dom\c\h \sigma=\dom\c(\gamma\c -)\c\h\zeta =(\widehat{\gamma}\c -)\c\dom\c\h\zeta,$$
	whence $\widehat{\gamma}$ needs to be the unique morphism $K\to E$ in $\E$ induced by $\dom\c\h\sigma$ via universality of $\dom\c \zeta$. We will now produce the inner $2$-cell $\gamma$ in $\E$ via the $2$-dimensional universality of $\dom\c\h \zeta$. Indeed $\gamma$ corresponds to a modification
	$$\Xi\:(t\c -)\c \dom\c\h \zeta\aMM{}(g\c \widehat{\gamma}\c -)\c \dom\c \h\zeta.$$
	Notice that the target of $\Xi$ coincides with $(g\c -)\c \dom\c\h \sigma$. Given $(A,X)\in \Groth{W}$, we will have that $\Xi_{(A,X)}=\gamma\ast \widehat{\zeta_{(A,X)}}$, and we want to obtain
	\begin{eqD*}
	\begin{cd}*[6][7]
		\dom(H(A,X)) \arrow[r,"{\widehat{\zeta_{(A,X)}}}"] \arrow[rd,"{H(A,X)}"'{inner sep=0.2ex},""{name=A}]\& K \arrow[r,"{\widehat{\gamma}}"] \arrow[d,"{t}"'{pos=0.3},""'{name=B},""{pos=0.43,name=C}] \& E \arrow[dl,"{g}",""'{name=D}] \\
		\& M
		\arrow[iso,from=A,to=B,"{\zeta_{(A,X)}}"'{inner sep=0.9ex},shift left=1.9ex,start anchor={[xshift=-1ex]}]
		\arrow[Rightarrow,from=C,to=D,"{\gamma}",shorten <=0.5ex]
	\end{cd}=
	\tr*[3.6][5.1][-3][1.1][l][-0.3][\sigma_{(A,X)}][1.15]{\dom(H(A,X))}{E}{M}{\widehat{\sigma_{(A,X)}}}{H(A,X)}{g}
	\end{eqD*}
	The component $\zeta_{(A,X)}$ is indeed a triangle filled with an isomorphism, by \remx\ref{excartinlaxslice}, since $\zeta$ is cartesian. Whence we need to take
	$$\Xi_{(A,X)}\deq\begin{cd}*[6.5][5.5]
		\& \dom(H(A,X))\arrow[r,"{\widehat{\sigma_{(A,X)}}}"]\arrow[ld,bend right=15,"{\widehat{\zeta_{(A,X)}}}"'{inner sep =0.2ex}]\arrow[d,"{H(A,X)}"{pos=0.3,description},""{name=A}] \& E \arrow[dl,bend left=15,"{g}",""'{pos=0.405,name=B}]\\
		K \arrow[r,"{t}"'] \arrow[r,iso,shift left=4.7ex,"{\zeta_{(A,X)}^{-1}}"'{inner sep=1ex},end anchor={[xshift=4.6ex]}] \& M
		\arrow[Rightarrow,"{\sigma_{(A,X)}}"{pos=0.64},from=A,to=B,shorten <=2.2ex,shorten >=1.3ex]
	\end{cd}$$
	It is straightforward to check that $\Xi$ is a modification between cartesian-marked oplax $2$-cocones. So $\Xi$ induces a unique $2$-cell $\gamma\:t \aR{} g\c \widehat{\gamma}$ in $\E$ such that $(\gamma\ast -)\ast (\dom\c\h \zeta)=\Xi$. It is easy to check that $\sigma=(\gamma\c -)\c \zeta$ (as cartesian-marked oplax natural transformations).
	
	\noindent We now show the uniqueness of $\gamma$. So assume there is some $\gamma'\:t\to g$ in $\laxslice{\E}{M}$ such that $\sigma=(\gamma'\c -)\c \zeta$. Then $\widehat{\gamma'}=\widehat{\gamma}$ by the argument above. Since the two $2$-cells $\gamma$ and $\gamma'$ in $\E$ are then between the same $1$-cells, it suffices to prove that
	$$(\gamma\ast -)\ast (\dom\c\h\zeta)=(\gamma'\ast -)\ast (\dom\c\h\zeta).$$
	But this can be checked on components, and $\gamma\ast \widehat{\zeta_{(A,X)}}=\gamma'\ast \widehat{\zeta_{(A,X)}}$ holds because
	$$(\gamma\c -)\c \zeta=\sigma=(\gamma'\c -)\c \zeta$$
	and $\zeta$ is cartesian (essentially, both give the same $\Xi$).
	
	It remains to prove the $2$-dimensional universality of $\zeta$. Given $g\:E\to M$, two morphisms $\gamma,\gamma'\:t\to g$ in $\laxslice{\E}{M}$ and a modification
	$$\Sigma\:(\gamma\c -)\c\h \zeta\aM{}(\gamma'\c -)\c\h\zeta\:\Delta 1\aoplaxn{}\HomC{\laxslice{\E}{M}}{H(-)}{g},$$
	we need to produce a $2$-cell $\Gamma\:\gamma\to \gamma'$ in $\laxslice{\E}{M}$ such that $\Sigma=(\Gamma\ast -)\ast \zeta$. But the latter equality is satisfied precisely when it is satisfied after composing with $\dom$. So consider $\dom\ast \Sigma$; as $\dom\c\h \zeta$ is universal, we find a unique $\widehat{\Gamma}\:\widehat{\gamma}\aR{}\widehat{\gamma'}$ such that
	$$\dom\ast \Sigma=(\widehat{\Gamma}\ast -)\ast (\dom\c\h \zeta).$$
	It is straightforward to show that $\widehat{\Gamma}$ gives a $2$-cell $\Gamma\:\gamma\to \gamma'$ in $\laxslice{\E}{M}$. By construction $\Sigma=(\Gamma\ast -)\ast \zeta$, and our argument has proved the uniqueness of $\Gamma$ as well.
\end{proof}

	We now conclude that $\dom\:\laxslice{\E}{M}\to \E$ is a $2$-colim-fibration. The explicit construction of liftings along $\dom$ then implies the conclusion of \conx\ref{consfirstapproach} (first approach) from this abstract point of view.

\begin{teor}\label{teordomis2colimfib}
	Let $\E$ be a $2$-category and $M\in \E$. Then $\dom\:\laxslice{\E}{M}\to \E$ is a $2$-colim-fibration. As a consequence, in the notation of \conx\ref{consfirstapproach}, 
	$$\fib{\wcolim{W}{F}}{q}{M}=\fib{\oplaxncolim{\left(F\c \groth{W}\right)}}{q}{M}=\oplaxncolim{L^q}$$
	in the lax slice $\laxslice{\E}{M}$. Here, $L^q$ is the $2$-diagram in $\laxslice{\E}{M}$ that corresponds to the cartesian-marked oplax $2$-cocone $\lambda^q$ on $M$ associated to the weighted $2$-cocylinder on $M$ that $q$ represents.
\end{teor}
\begin{proof}
	Putting together \remx\ref{remjustifydomlaxslice} and \thex\ref{teordomreflects}, we obtain that $\dom\:\laxslice{\E}{M}\to \E$ is a 2-colim-fibration, by \prox\ref{propcharacterization2colimfib}.

	We prove the second part of the statement. Calling $\theta$ the universal cartesian-marked oplax $2$-cocone that exhibits $C=\oplaxncolim{(F\c\groth{W})}$, we obtain that the lifting of $(F\c\groth{W},\theta)$ through $\dom$ to $q$ (calculated as in \conx\ref{consliftoplaxnormalcocone})
	\begin{cd}[1.5][0.5]
		\ov{F\c\groth{W}}(A,X)\arrow[rrd,bend left=20,"{\ov{\theta}_{(A,X)}}"{inner sep=0.2ex},""'{name=J}]\arrow[dr,"{\ov{F\c\groth{W}}(f,\alpha)}"',""{name=A}]\&[-0.6ex]\\
		\& \ov{F\c\groth{W}}(B,X')\arrow[Rightarrow,from=J,"{\ov{\theta}_{f,\alpha}}"{pos=0.34},shorten <=0.75ex,shorten >=0.65ex]\arrow[dd,mapsto,"{\dom}"{pos=0.20},shorten <=0.15ex,shorten >=7.15ex,shift right=5.7ex]\arrow[r,"{\ov{\theta}_{(B,X')}}"']\&[10ex] q\arrow[dd,mapsto,"{\dom}"{pos=0.51},shorten <=3.75ex,shorten >=3.4ex]\\[3ex]
		{F}(A)\arrow[rrd,bend left=20,"{{\theta}_{(A,X)}}"{inner sep=0.2ex},""'{name=K}]\arrow[dr,"{{F}(f)}"',""{name=B}]\\
		\& {F}(B)\arrow[Rightarrow,from=K,"{\theta_{f,\alpha}}"{pos=0.34},shorten <=0.75ex,shorten >=0.65ex]\arrow[r,"{{\theta}_{(B,X')}}"']\&[10ex] C
	\end{cd}
	exhibits
	$$q=\oplaxncolim{\ov{F\c\groth{W}}}$$
	in $\laxslice{\E}{M}$. And we can calculate $\ov{F\c \groth{W}}$ and $\ov{\theta}$ explicitly, looking at the action of $\y{M}\:\E\op\to \CAT$ on $1$-cells and $2$-cells, since $\dom=\groth{\y{M}}$. Given $(A,X)\in \Groth{W}$,
	$$\ov{F\c\groth{W}}(A,X)=\y{M}(\theta_{(A,X)})(q)=q\c \theta_{(A,X)}=\lambda^q_{(A,X)}=L^q(A,X)$$
	\begin{eqD*}
		\ov{\theta}_{(A,X)}=\tr*[3.6][5.1][0][0][e][-0.3][\id{}][1.33]{F(A)}{C}{M}{\theta_{(A,X)}}{\lambda^q_{(A,X)}}{q}
	\end{eqD*}
	Given $(f,\alpha)\:(A,X)\to (B,X')$ in $\Groth{W}$,
	$$\ov{\theta}_{f,\alpha}=\theta_{f,\alpha}\:(\theta_{(A,X)},\id{})\to (\theta_{B,X'}\c F(f),\y{M}(\theta_{f,\alpha})_q)$$
	whence, since $\y{M}(\theta_{f,\alpha})_q=q\ast \theta_{f,\alpha}=\lambda^q_{f,\alpha}$,
	\begin{eqD*}
		\ov{F\c\groth{W}}(f,\alpha)=\tr*[3.6][5.1][0][0][l][-0.3][\lambda^q_{f,\alpha}][0.85]{F(A)}{F(B)}{M}{F(f)}{\lambda^q_{(A,X)}}{\lambda^q_{(B,X')}}=L^q(f,\alpha).
	\end{eqD*}
	Given $\delta\:(f,\alpha)\to(g,\beta)\:(A,X)\to (B,X')$ in $\Groth{W}$, $\ov{F\c\groth{W}}(\delta)=F(\delta)=L^q(\delta)$. 
\end{proof}

\begin{rem}\label{remreflectioninsideFcattheory}
	For $\dom\:\laxslice{\E}{M}\to \E$, that is a cloven discrete 2-fibration, reflecting all the cartesian 2-colimits is equivalent to reflecting all the cloven 2-colimits, by \prox\ref{propcharacterization2colimfib}) (but it requires the same effort to prove one or the other).

	As $\laxslice{\E}{M}=\Groth{\y{M}}$, the lax slice has a canonical $\F$-category structure (see \recx\ref{rectwocatofel}), with loose morphisms the usual ones and tight morphisms the triangles filled with an identity. That is, the tight part of $\laxslice{\E}{M}$ is the strict $2$-slice $\slice{\E}{M}$.

	The 2-functor $L^q$ produced in \thex\ref{teordomis2colimfib} is then an $\F$-functor, with respect to this canonical $\F$-category structure. Moreover, the cartesian-marked oplax conical 2-colimit $q=\oplaxncolim{L^q}$ is actually a strict/oplax conical $\F$-colimit. Indeed, a universal strict/oplax $\F$-cocone over $q$ in $\laxslice{\E}{M}$ is the same as a universal cloven (cartesian-marked oplax) 2-cocone with the extra property of jointly detecting tightness (see \prox\ref{charactsoplaxFcolimit} and \recx\ref{rectwocatofel}). And the universal cloven 2-cocone $\ov{\theta}$ produced in the proof of \thex\ref{teordomis2colimfib} also has such extra property of detecting tightness, by universality of $\dom\c \ov{\theta}=\theta$.

	The same considerations show that the $\F$-functor $\dom\:\laxslice{\E}{M}\to \E$ (taking all morphisms in $\E$ to be tight) reflects strict/oplax conical $\F$-colimits, in the sense that it reflects the universality of cloven 2-cocones with the property of jointly detecting tightness. And furthermore, that it was actually enough to prove this reflection result rather than \thex\ref{teordomreflects} directly, to deduce that $\dom$ is a 2-colim-fibration and thus preserves all cartesian 2-colimits. 

	We can thus view \thex\ref{teordomreflects} and \thex\ref{teordomis2colimfib} inside the framework of $\F$-category theory. Such a framework will be crucial in the following sections, to establish results of lifting and preservation for $\dom\:\laxslice{\E}{M}\to \E$, as well as for change of base 2-functors between lax slices.
\end{rem}

\begin{rem}
	The proofs of \thex\ref{teordomreflects} and \thex\ref{teordomis2colimfib} equally work for sigma colimits of Descotte, Dubuc and Szyld's~\cite{descottedubucszyld_sigmalimandflatpseudofun} with respect to the cartesian marking, in the place of cartesian-marked oplax conical 2-colimits. Slightly modified, they work as well for sigma bicolimits. Exactly as every weighted 2-colimit can be reduced to a cartesian-marked oplax conical one, every (weighted) pseudo colimit can be reduced to a sigma colimit, and every bicolimit can be reduced to a sigma bicolimit (see \cite{descottedubucszyld_sigmalimandflatpseudofun}, where the result is derived from Street's~\cite{street_limitsindexedbycatvalued}). We thus have the following corollary.
\end{rem}
	
\begin{coroll}
	$$\fib{\pseudocolim{W}{F}}{q}{M}=\fib{\sigmacolim{\left(F\c \groth{W}\right)}}{q}{M}=\sigmacolim{L^q}$$
	$$\fib{\bicolim{W}{F}}{q}{M}=\fib{\sigmabicolim{\left(F\c \groth{W}\right)}}{q}{M}=\sigmabicolim{L^q}$$
	in the lax slice $\laxslice{\E}{M}$, where $L^q$ is the $2$-diagram in $\laxslice{\E}{M}$ that corresponds to the sigma cocone $\lambda^q$ on $M$ associated to the weighted pseudo cocylinder on $M$ that $q$ represents.
\end{coroll}
\begin{proof}
	\conx\ref{consliftoplaxnormalcocone} equally works to lift any oplax cocone $\theta$ to an oplax cocone $\ov{\theta}$. If $\theta$ is sigma natural, then $\ov{\theta}$ is sigma natural as well, since a discrete 2-fibration lifts isomorphic 2-cells to isomorphic 2-cells. The same argument of the proof of \thex\ref{teordomreflects} then shows that $\dom$ also reflects the universality of cartesian sigma cocones.
	
	Slightly modified, the proof of \thex\ref{teordomreflects} works as well for sigma bicolimits. Indeed, via bi-universality of $\dom\c \zeta$, we have that $\dom\c \sigma$ induces $\widehat{\gamma}$ and an isomorphic modification $\kappa\:(\widehat{\gamma}\c -)\c \dom\c \zeta\iso \dom\c \sigma$. We then induce the 2-cell $\gamma$ from a slightly modified version of $\Xi$ obtained by pasting the assignment of $\Xi_{(A,X)}$ in the proof of \thex\ref{teordomis2colimfib} with $\kappa_{(A,X)}^{-1}$. So that $\kappa_{(A,X)}$ is by construction a 2-cell in the lax slice, and hence becomes a modification between sigma cocones (as the condition of modification holds after applying $\dom$).
\end{proof}

\section{Lifting and preservation of colimits for the domain 2-functor}\label{sectionliftingpreservfordom}

In this section, we generalize to dimension $2$ the bijective correspondence between cocones on $M$ and diagrams in the slice over $M$ (\prox\ref{correspoplaxncoconesdiag}), thanks to enhanced (or $\F$-)category theory, introduced in Lack and Shulman's~\cite{lackshulman_enhancedtwocatlimlaxmor}. We then prove results of lifting and preservation of $2$-colimits for $\dom\:\laxslice{\E}{M}\to \E$ (\prox\ref{proplifting} and \thex\ref{teordomhassemilaxFadjoint}), establishing a full $2$-categorical generalization of \thex\ref{colimitsinslicesdim1}. Such results are not about all (cartesian-marked oplax) $2$-colimits, but this makes sense from an $\F$-categorical point of view. Furthermore, we characterize which 2-colimits in the lax slice come from 2-colimits in the base 2-category via the work of Section~\ref{sectioncolimitsin2slices}.

In dimension $1$, assuming products in $\C$, the functor $\dom\:\slice{\C}{M}\to M$ preserves colimits because it has a right adjoint, namely $\dom \dashv M\x -$. In dimension $2$, assuming products in $\E$, the $2$-functor $M\x -$ is only a lax right adjoint to $\dom\:\laxslice{\E}{M}\to \E$. As explained in Section~\ref{sectionlaxfadjpreserv}, this would not be enough to guarantee that $\dom$ preserves 2-colimits. But we prove that the lax adjunction between $\dom\:\laxslice{\E}{M}\to \E$ and $M\x -$ is a right-semi-lax (tight) $\F$-adjunction (\thex\ref{teordomhassemilaxFadjoint}). The work of Section~\ref{sectionlaxfadjpreserv} then ensures that $\dom$ preserves all tight strict/oplax $\F$-colimits, as well as a large class of 2-colimits.

\begin{rem}\label{rephrasecorrespondence}
	We have already seen in \conx\ref{consfirstapproach} (first approach) that, in dimension $2$, we can reorganize a cartesian-marked oplax $2$-cocone on $M$ as a $2$-diagram in the lax slice $\laxslice{\E}{M}$. Indeed this was the main idea, together with the reduction of weighted $2$-colimits to cartesian-marked oplax conical ones, of the first approach to the categorification of the equivalence between colimits in $1$-dimensional slices and maps from the colimit of the domains. However, not every $2$-diagram in $\laxslice{\E}{M}$ can produce a cartesian-marked oplax $2$-cocone on $M$, as the cartesian-marked condition may fail.
	
	This can be explained in the framework of $\F$-category theory. As we said in \remx\ref{remreflectioninsideFcattheory}, the lax slice has a canonical $\F$-category structure, with its tight part given by the strict $2$-slice. And we have already noticed there that the 2-functor $L^q$ of \thex\ref{teordomis2colimfib} is an $\F$-diagram. 
\end{rem}

\begin{prop}\label{correspoplaxncoconesdiag}
	Let $\E$ be a $2$-category and $M\in \E$. Consider then a marking $W\:\A\op\to \CAT$ with $\A$ small and a $2$-diagram $D\:\Groth{W}\to\E$. There is a bijection between cartesian-marked oplax $2$-cocones
	$$\lambda\:\Delta 1\aoplaxn{}\HomC{\E}{D(-)}{M}$$
	on $M$, that are the same as marked oplax 2-cocones with respect to the canonical $\F$-category structure on $\Groth{W}$, and $\F$-diagrams $\ov{D}\:\Groth{W}\to \laxslice{\E}{M}$, i.e.\ 2-diagrams that send every morphism of type $(f,\id{})$ to a triangle filled with an identity.
\end{prop}
\begin{proof}
	Given $\lambda$, since $\dom$ is a discrete $2$-fibration, we can lift $(D,\lambda)$ to $\id{M}$ and obtain a pair $(\ov{D},\ov{\lambda})$ as in \conx\ref{consliftoplaxnormalcocone}. Exactly as in the proof of \thex\ref{teordomis2colimfib}, we can calculate $\ov{D}$ explicitly, obtaining the formulas
	$$\ov{D}(f,\alpha)=(D(f,\alpha),\lambda_{f,\alpha})\:\lambda_{(A,X)}\to \lambda_{(B,X')}$$
	
	Starting instead from an $\F$-diagram $\ov{D}$, we can reorganize its data as a cartesian-marked oplax $2$-cocone $\lambda$ for $\dom\c \ov{D}$ on $M$, with formulas
	$$\lambda_{(A,X)}\deq \ov{D}(A,X)\:\dom(\ov{D}(A,X))\to M.$$
	$$\lambda_{f,\alpha}\deq \ov{D}(f,\alpha).$$
	The $2$-functoriality of $\ov{D}$ guarantees that $\lambda$ is oplax natural. And given $(f,\id{})$ in $\Groth{W}$, we obtain $\lambda_{f,\id{}}=\ov{D}(f,\id{})=\id{}$.
	
	It is clear that the two constructions we have produced are inverses of each other.
\end{proof}

We can now prove a result of lifting of $2$-colimits for $\dom\:\laxslice{\E}{M}\to \E$.

\begin{prop}\label{proplifting}
	Let $\E$ be a $2$-category and let $M\in \E$. Then the $2$-colim-fibration $\dom\:\laxslice{\E}{M}\to \E$ lifts all the cartesian-marked oplax conical $2$-colimits of $\F$-diagrams.
	
	That is, given a marking $W\:\A\op\to \CAT$ with $\A$ small, an $\F$-diagram $H\:\Groth{W}\to \laxslice{\E}{M}$ and a universal cartesian-marked oplax $2$-cocone
	$$\theta\:\Delta 1\aoplaxn{}\HomC{\E}{(\dom\c H)(-)}{C}$$
	that exhibits $C=\oplaxncolim{(\dom\c H)}$ in $\E$, there exist $q\in \laxslice{\E}{M}$ over $C$ and a universal cartesian-marked oplax $2$-cocone
	$$\ov{\theta}\:\Delta 1\aoplaxn{}\HomC{\laxslice{\E}{M}}{ H(-)}{q}$$
	for $H$ on $q$ exhibiting $q=\oplaxncolim{H}$ in $\laxslice{\E}{M}$ such that $\dom\c \h\ov{\theta}=\theta$.
\end{prop}
\begin{proof}
	By \prox\ref{correspoplaxncoconesdiag} (together with its proof), the $\F$-diagram $H\:\Groth{W}\to\laxslice{\E}{M}$ corresponds to a cartesian-marked oplax $2$-cocone $\lambda$ for $\dom\c H$ on $M$. As $\theta$ is universal, then $\lambda$ induces a unique morphism $q\:C\to M$ such that $\lambda=(q\c -)\c \theta$.
	
	Since $\dom$ is a 2-colim-fibration, by \thex\ref{teordomis2colimfib}, we can lift the pair $(\dom\c H, \theta)$ to $q$, obtaining via \conx\ref{consliftoplaxnormalcocone} a pair $(\ov{\dom\c H},\ov{\theta})$ over $(\dom\c H,\theta)$ exhibiting
	$$q=\oplaxncolim{\ov{\dom\c H}}.$$
	And we can calculate $\ov{\dom\c H}$ explicitly, as in the proof of \thex\ref{teordomis2colimfib}. Given $\delta\:(f,\alpha)\aR{}(g,\beta)\:(A,X)\to (B,X')$ in $\Groth{W}$
	$$\ov{\dom\c H}(A,X)=q\c \theta_{(A,X)}=\lambda_{(A,X)}=H(A,X)$$
	$$\ov{\dom\c H}(f,\alpha)=(\dom(H(f,\alpha)),q\ast \theta_{f,\alpha})=(\dom(H(f,\alpha)),\lambda_{f,\alpha})=H(f,\alpha)$$
	$$\ov{\dom\c H}(\delta)=\dom(H(\delta))=H(\delta)$$
	So $\ov{\dom\c H}=H$ and we conclude.
\end{proof}

\begin{rem}
	In order to achieve a full 2-categorical generalization of \thex\ref{colimitsinslicesdim1}, it only remains to address preservation of 2-colimits for $\dom\:\laxslice{\E}{M}\to \E$. The strategy to reach such a result will be to prove that $\dom$ has a strict right-semi-lax (tight) right $\F$-adjoint, assuming products in $\E$. We have not investigated results of preservation of particular colimits for $\dom\:\laxslice{\E}{M}\to \E$ when $\E$ lacks products.

	As in dimension 1, the right adjoint to $\dom$ will be $M\x -$. But since we need to consider the lax slice, we can only find a lax adjunction. Although a lax adjunction would not be enough to guarantee preservation of colimits, a strict right-semi-lax (tight) $\F$-adjunction is enough, as proved in Section~\ref{sectionlaxfadjpreserv}. Notice that $\dom\:\laxslice{\E}{M}\to \E$ has a structure of $\F$-functor, taking all morphisms in $\E$ to be tight.

	The obtained result of preservation of 2-colimits is expressed in an $\F$-categorical language, but remember that any weighted 2-colimit can be reduced to such $\F$-categorical context. See the following results of this section for what the conditions of \thex\ref{teordomhassemilaxFadjoint} mean explicitly in practice.
\end{rem}

\begin{teor}\label{teordomhassemilaxFadjoint}
	Let $\E$ be a $2$-category with products and let $M\in \E$. Then the $\F$-functor $\dom\:\laxslice{\E}{M}\to \E$ has a strict right-semi-lax \pteor{tight} right $\F$-adjoint, given by $M\x -$.
	
	As a consequence, by \thex\ref{teorsemilaxFadjpreserve}, $\dom$ preserves all tight strict/oplax $\F$-colimits, but also all the universal marked oplax $2$-cocylinders for an $\F$-diagram which have tight $\lambda$-components.
\end{teor}
\begin{proof}
	We use the universal mapping property \prox\ref{univmappingproplaxadj} that characterizes a lax adjunction to build a right-semi-lax right adjoint $U$ to $\dom\:\laxslice{\E}{M}\to \E$. For every $E\in \E$, we define $U(E)\deq (M\x E\ar{\pr{1}}E)$ and $\eps_E\:M\x E\ar{\pr{2} }E$, that is tight in $\E$, remembering that in dimension $1$ the domain functor from $\slice{\C}{M}$ is left adjoint to $M\x -$.
	
	We show that such counit is universal in the lax sense. Given $h\:\dom(K\ar{t} M)\to E$ in $\E$, take $\ov{h}\deq ((t,h),\id{})\:(K\ar{t} M)\to (M\x E\ar{\pr{1}}M)$, which is tight in $\laxslice{\E}{M}$ (see \remx\ref{rephrasecorrespondence}), and $\lambda_h\deq \id{}$. 
	\begin{cd}[5][7]
		\dom(K\aar{t}M) \arrow[rd,bend left=25,"{h}",""'{name=C}]\arrow[d,"{\dom((t,h),\id{})}"'] \\
		\dom(M\x E\aar{\pr{1}}M)\arrow[r,"{\pr{2}}"'] \arrow[equal,from=C,"{\lambda_h}"{pos=0.54},shorten <=3.6ex,shorten >= 3.4ex]\& E
	\end{cd}
	This guarantees that we will find a right-semi-lax adjunction in the end (see \prox\ref{univmappingproplaxadj}). Given then another morphism
	\tr[4][5][0][-2][l][-0.3][\gamma][0.85]{K}{M\x E}{M}{\widehat{\gamma}}{t}{\pr{1}}
	in $\laxslice{\E}{M}$ and another $\sigma\:h\aR{}\pr{2}\c\h \widehat{\gamma}$ in $\E$, there is a unique $\delta\:((t,h),\id{})\aR{}(\widehat{\gamma},\gamma)$ in $\laxslice{\E}{M}$ such that
	$$\begin{cd}*[6][8]
		\dom(K\aar{t}M) \arrow[rd,bend left,"{h}",""'{name=C}]\arrow[d,bend left=20,"{\dom((t,h),\id{})}"{pos=0.44},""'{name=A}]\arrow[d,bend right=60,"{\dom(\widehat{\gamma},\gamma)}"'{pos=0.59},""{name=B},shorten <=-0.8ex,shorten >=0.2ex] \\
		\dom(M\x E\aar{\pr{1}}M)\arrow[r,"{\pr{2}}"'] \arrow[equal,from=C,"{\lambda_h}"{pos=0.54},shorten <=4ex,shorten >= 3.7ex]\& E
		\Arb[inner sep=0.55ex]{Rightarrow,shift right=0.2ex,pos=0.47}{\delta}{A}{B}{0.2}
	\end{cd}=\h[4]\begin{cd}*[6][8]
		\dom(K\aar{t}M) \arrow[rd,bend left,"{h}",""'{name=C}]\arrow[d,"{\dom(\widehat{\gamma},\gamma)}"'] \\
		\dom(M\x E\aar{\pr{1}}M)\arrow[r,"{\pr{2}}"'] \arrow[Rightarrow,from=C,"{\sigma}",shorten <=1.7ex,shorten >= 1.9ex]\& E
	\end{cd}\v[2]$$
	Indeed $\delta$ is determined by $\dom(\delta)$, which is $\tc*[7][30][pos=0.53][pos=0.54][][0.3]{K}{M\x E}{(t,h)}{\widehat{\gamma}}{\delta}$,
	that needs to satisfy $\pr{1}\ast \delta=\gamma$ in order to be a $2$-cell $((t,h),\id{})\aR{}(\widehat{\gamma},\gamma)$ and $\pr{2}\ast \delta=\sigma$ by the condition above. So $\delta$ needs to be $(\gamma,\sigma)$, and this works.
	
	We then see that, for every $E\in \E$, $\ov{\eps_{E}}=\id{}$ since in this case $(t,h)=(\pr{1},\pr{2})=\id{}$ (and $\lambda$ is always the identity). Moreover, for every $h\:\dom(K\ar{t} M)\to E$ in $\E$,
	$$\ov{h\c\h\eps_{F(A)}}\c \ov{\id{F(A)}}=((\pr{1},h\c \pr{2}),\id{})\c ((t,\id{}),\id{})=((t,h),\id{})=\ov{h},$$
	making the assumption of equation \refs{assumptioncharactlaxadj} (of \prox\ref{univmappingproplaxadj}) hold.
	
	By \prox\ref{univmappingproplaxadj}, as $\lambda_h$ is always the identity, $U$ extends to an oplax functor, $\eps$ extends to a $2$-natural transformation and there exist a lax natural transformation $\eta$ and a modification $t$ such that $U$ is a right-semi-lax right adjoint to $\dom\:\laxslice{\E}{M}\to \E$. But it is easy to see, following the explicit construction of \prox\ref{univmappingproplaxadj}, that $U$ is the (strict) $2$-functor
	\begin{fun}
		M\x - & \: & \E \hphantom{...}& \too &\hphantom{..} \laxslice{\E}{M} \\[1ex]
		& & E\hphantom{...} & \mto & \p{M\x E\ar{\pr{1}}E} \\[1ex]
		& & E\ar{e} E' & \mto &\hphantom{..} (\id{}\x e,\id{})\\[0.8ex]
		& & e\aR{\beta}e' \hphantom{.}& \mto &\hphantom{..} \id{}\x \beta
	\end{fun}
	Then, for every $(K\ar{t}M)\in \laxslice{\E}{M}$,
	$$\eta_t=\h[4]\tr*[4][5][0][-2][e][-0.3][\id{}][1.2]{K}{M\x K}{M}{(t,\id{})}{t}{\pr{1}}$$
	The fact that $\ov{h}$ is always tight implies that $U=M\x -$ is an $\F$-functor and that $\eta$ has tight components. Given a morphism $(\widehat{\gamma},\gamma)\:(K\ar{t}M)\to (K'\ar{t'}M)$ in $ \laxslice{\E}{M}$,
	$$\eta_{(\widehat{\gamma},\gamma)}\h[2]=\quad\begin{cd}*[1.5][2.5]
		\& M\x K \arrow[rd,"{\id{}\x \widehat{\gamma}}"{inner sep=0.3ex}] \arrow[dd,Rightarrow,"{(\gamma,\id{})}",shorten <=1.25ex,shorten >=1.2ex]\\
		K\arrow[ru,"{(t,\id{})}"{inner sep=0.3ex}] \arrow[rd,"{\widehat{\gamma}}"']\& \& M\x K' \\
		\& K' \arrow[ru,"{(t',\id{})}"'{pos=0.35,inner sep=0.3ex}]
	\end{cd}$$
	whence it is clear that $\eta$ is (tight) strict/lax $\F$-natural (since $\eta_{(\widehat{\gamma},\id{})}=\id{}$). Finally, $t=\id{}$, giving a strict right-semi-lax adjunction. We have also already checked that $\dom\:\laxslice{\E}{M}\to \E$ and $M\x -$ are $\F$-functors, $\eta$ is (tight) strict/lax $\F$-natural and $\eps$ has tight components, giving a strict right-semi-lax (tight) $\F$-adjunction.
\end{proof}

By \recx\ref{rectwocatofel}, cartesian-marked oplax conical 2-colimits (and thus weighted 2-colimits alike) are inscribed in the $\F$-categorical context of universal marked oplax 2-cocylinders. We obtain the following.

\begin{coroll}\label{remexplicitpresdom}
	Let $\E$ be a $2$-category with products and let $M\in \E$. The 2-functor $\dom\:\laxslice{\E}{M}\to \E$ preserves all the universal cartesian-marked oplax $2$-cocones for an $\F$-diagram which have tight components.

	That is, given 2-functors $W\:\A\op\to \CAT$ with $\A$ small and $H\:\Groth{W}\to \laxslice{\E}{M}$ that sends every morphism of type $(f,\id{})$ to a triangle filled with an identity, if
	$$\zeta\:\Delta 1\aoplaxn{}\HomC{\laxslice{\E}{M}}{ H(-)}{q}$$
	is a universal cartesian-marked oplax $2$-cocone for $H$ on $q\in \laxslice{\E}{M}$ exhibiting $q=\oplaxncolim{H}$ such that $\zeta_{(A,X)}$ is a triangle filled with an identity for every $(A,X)\in \Groth{W}$, then $\dom\c \h\zeta$ is universal as well, exhibiting $$\dom(q)=\oplaxncolim{\left(\dom\c H\right)}.$$ 
\end{coroll}

\begin{rem}\label{remsharperpreservationdom}
	We can actually obtain a sharper result of preservation of $2$-colimits for $\dom\:\laxslice{\E}{M}\to \E$, as we show in \prox\ref{propsharperpreservationdom}. Namely, we can omit the assumption that the universal cartesian-marked oplax 2-cocones for an $\F$-diagram have tight $\lambda$-components. Indeed, in the proof of \thex\ref{teorsemilaxFadjpreserve}, the preservation of the universal marked oplax 2-cocylinder uses the assumption that the $\mu^\lambda_{A}(X)$'s are tight only to guarantee the uniqueness parts of the $1$- and $2$-universal property. But we can prove both uniqueness results in another way, taking advantage of the simple description of the strict right-semi-lax right $\F$-adjoint $U=M\x -$ of $\dom$.
\end{rem}

\begin{prop}\label{propsharperpreservationdom}
	Let $\E$ be a $2$-category with products and let $M\in \E$. The 2-functor $\dom\:\laxslice{\E}{M}\to \E$ preserves all the universal cartesian-marked oplax $2$-cocones for an $\F$-diagram \pteor{without assuming them to have tight $\lambda$-components}.
\end{prop}
\begin{proof}
	We only need to prove the uniqueness part of the $1$- and $2$-dimensional universal property, by \remx\ref{remsharperpreservationdom}. Following the proof of \thex\ref{teorsemilaxFadjpreserve} with $F=\dom$ and considering $\delta'\:\dom(K\ar{t} M)\to Z$ in $\E$ such that $(\delta'\c -)\c \dom\c \mu=\sigma$, rather than considering $S(\delta')=((t,\delta'),\id{})$, we define
	$$\gamma'\deq \h[4]\tr*[5][5][0][-2][l][-0.3][\gamma][0.85]{K}{M\x Z}{M}{(\pr{1}\c \dom(\gamma),\delta')}{t}{\pr{1}}$$
	Then $\gamma'$ satisfies $(\gamma'\c -)\c \mu=S\c \sigma$, by the universal property of the product, since $\gamma$ satisfies the analogous equation and $(\delta'\c -)\c \dom\c \mu=\sigma$. By the uniqueness of $\gamma$, we obtain that $\gamma'=\gamma$ and hence $\delta'=\pr{2}\c \dom(\gamma)=T(\gamma)=\delta$. 
	
	Analogously, we can prove also the uniqueness of the $2$-dimensional universal property, producing from $\Delta'$ the $2$-cell $(\pr{1}\ast \dom(\Gamma),\Delta')$ between the two suitable triangles $\gamma'$ here above. We indeed obtain $\Delta'=\pr{2}\ast \dom(\Gamma)=T(\Gamma)=\Delta$.
\end{proof}

We now shed more light on the $\F$-categorical assumptions of \prox\ref{propsharperpreservationdom}, \thex\ref{teordomhassemilaxFadjoint} and \prox\ref{proplifting}. At the same time, we characterize which 2-colimits in the lax slice come from 2-colimits of the base 2-category via the work of Section~\ref{sectioncolimitsin2slices}.

\begin{teor}\label{teorcolimitsinlaxslicesthatcomefromthebase}
	Let $\E$ be a 2-category with products and let $q\:C\to M$ in $\E$. The following are equivalent:
	\begin{enumT}
		\item $q$ can be expressed as a 2-colimit in $\laxslice{\E}{M}$ that comes from a \pteor{weighted or cartesian-marked oplax conical} 2-colimit in $\E$ via the work of Section~\ref{sectioncolimitsin2slices} \pteor{see \conx\ref{consfirstapproach} or \thex\ref{teordomis2colimfib}};
		\item $q$ can be expressed as a strict/oplax conical $\F$-colimit of an $\F$-diagram $\Groth{W}\to\laxslice{\E}{M}$;
		\item $q$ can be expressed as the cartesian-marked oplax conical 2-colimit in $\laxslice{\E}{M}$ of an $\F$-diagram, with the universal 2-cocone having tight components;
		\item $q$ can be expressed as the cartesian-marked oplax conical 2-colimit in $\laxslice{\E}{M}$ of an $\F$-diagram;
		\item $q$ can be expressed as a \pteor{weighted or cartesian-marked oplax conical} 2-colimit in $\laxslice{\E}{M}$ that is preserved by $\dom\:\laxslice{\E}{M}\to \E$.
	\end{enumT}
\end{teor}
\begin{proof}
	$(i)\aR{}(ii)$ holds by \remx\ref{remreflectioninsideFcattheory}. $(ii)\aR{}(iii)\aR{}(iv)$ is trivial. $(iv)\aR{}(v)$ holds by \prox\ref{propsharperpreservationdom}.

	We show $(v)\aR{}(i)$. By assumption, the domain of $q$ needs to be a 2-colimit in $\E$. We then conclude by \thex\ref{teordomis2colimfib}.
\end{proof}

\begin{rem}\label{exacolimitsthatcomefromthetwocat}
	By \thex\ref{teorcolimitsinlaxslicesthatcomefromthebase}, all 2-colimits in $\laxslice{\E}{M}$ that come from 2-colimits in $\E$ via the work of Section~\ref{sectioncolimitsin2slices} satisfy the conditions of \prox\ref{propsharperpreservationdom}, and actually also of \corx\ref{remexplicitpresdom}.

	Moreover, we will see in \exax\ref{exacolimthatcomefromthetwocattaust} that applying the change of base 2-functors of Section~\ref{sectionchangeofbase} to such 2-colimits gives again 2-colimits that satisfy the conditions of \prox\ref{propsharperpreservationdom}. This means that the conditions of \prox\ref{propsharperpreservationdom} are satisfied by most of the colimits we have in practice. For example, all the colimits in 2-dimensional slices that we needed in order to develop our~\cite{mesiti_twoclassifiersdensegenstacks} satisfy such conditions.

	\thex\ref{teorcolimitsinlaxslicesthatcomefromthebase} also shows that our result of preservation of 2-colimits for $\dom\:\laxslice{\E}{M}\to \E$ is sharp, at least when $\E$ has products.

	Notice that, as opposed to what happens in dimension $1$, there is more space in the 2-dimensional $\laxslice{\E}{M}$ allowing the presence of worse-behaved colimits that do not come from $\E$. After all, the domain 2-functor now only has a nice lax right adjoint rather than a strict one.
\end{rem}

\begin{exampl}\label{exaspecificcolimitsinlaxslice}
	In practice, most of the 2-colimits in $\laxslice{\E}{M}$ arise from given 2-colimits in $\E$, via the work of Section~\ref{sectioncolimitsin2slices}. But it is also interesting to ask whether a specific shape of 2-colimits in $\laxslice{\E}{M}$ comes from 2-colimits in $\E$.

	All coproducts, copowers, coequalizers of tight morphisms, coinserters of tight morphisms and coequifiers of 2-cells between tight morphisms in $\laxslice{\E}{M}$ can be expressed as cartesian-marked oplax conical 2-colimits of an $\F$-diagram. Indeed, it is straightforward to check that the essential conicalization of such weighted 2-colimits, as described in \recx\ref{rectwocatofel}, gives an $\F$-diagram $F\c \groth{W}$. Conical 2-colimits, such as coproducts and coequalizers, are readily seen to also be cartesian-marked oplax conical 2-colimits, with $\groth{W}=\id{}$ and all morphisms in $\Groth{W}$ being tight. For copowers, identities are the only tight morphisms in the corresponding $\Groth{W}$.

	So all coproducts, copowers, coequalizers of tight morphisms, coinserters of tight morphisms and coequifiers of 2-cells between tight morphisms in $\laxslice{\E}{M}$ satisfy the equivalent conditions of \thex\ref{teorcolimitsinlaxslicesthatcomefromthebase}.

	By the same argument, thanks to \prox\ref{proplifting}, if the base 2-category has all 2-colimits of any of the shapes above, the lax slice has all 2-colimits of the same shape as well, calculated as in the proof of \prox\ref{proplifting}.
\end{exampl}

\section{Change of base between lax slices}\label{sectionchangeofbase}

In dimension $1$, the concept of change of base between slices is definitely helpful. And it is well-known that the pullback perfectly realizes such a job. For $\CAT$, given a functor $\tau\:\E\to \B$, we can still consider the pullback $2$-functor $\tau\st\:\slice{\CAT}{\B}\to \slice{\CAT}{\E}$ between strict slices. And it is well-known that such change of base functor has a right 2-adjoint $\tau\stb$, and thus preserves all weighted 2-colimits, precisely when $\tau$ is a Conduch\'{e} functor.

However, Section~\ref{sectioncolimitsin2slices} showed that, in order to generalize the calculus of colimits in $1$-dimensional slices to dimension $2$, one needs to consider lax slices. And it is then very helpful to have a change of base $2$-functor between lax slices of a finitely complete $2$-category. Taking general pullbacks does not produce a functor between lax slices. We believe that the most natural way to achieve a change of base 2-functor between lax slices is to take comma objects. Equivalently, we can take pullbacks along split Grothendieck opfibrations (that serve as a kind of fibrant replacement), see \prox\ref{propreplacement}. Such a point of view is preferable in the context of this section, since Grothendieck opfibrations in $\CAT$ are always Conduch\'{e} and we can generalize the ideas for finding a right adjoint to the pullback functor $\tau\st\:\slice{\CAT}{\B}\to \slice{\CAT}{\E}$ (see Conduch\'{e}'s~\cite{conduche_existadjointsadroitcondfunc}) to lax slices.

We take Street's~\cite{street_fibandyonlemma} and Weber's~\cite{weber_yonfromtwotop} as main references for Grothendieck opfibrations in a general $2$-category. We prove that if $\tau\:\E\to \B$ is a split Grothendieck opfibration in a $2$-category $\K$, then pulling back along $\tau$ extends to a $2$-functor $\tau\st\:\laxslice{\K}{\B}\to \laxslice{\K}{\E}$.

Furthermore, we prove that when $\K=\CAT$, the $2$-functor $\tau\st$ between lax slices has a strict right-semi-lax loose right $\F$-adjoint. This generalizes Palmgren's~\cite{palmgren_groupoidsloccartclos}, that proved a similar result for pseudoslices of groupoids, from the comma objects point of view. As a consequence to such adjunction result, by \thex\ref{teorsemilaxFadjpreserve}, $\tau\st\:\laxslice{\CAT}{\B}\to \laxslice{\CAT}{\E}$ preserves all the universal marked oplax $2$-cocylinders for an $\F$-diagram which have tight $\lambda$-components. Remember that the general context of universal marked oplax $2$-cocylinders includes the one of weighted $2$-colimits, after reducing them to cartesian-marked oplax conical ones. We then extend this result of preservation of $2$-colimits for $\tau\st$ to more general $2$-categories other than $\CAT$: firstly to prestacks (\prox\ref{proptaustprestacks}) and then to any finitely complete $2$-category with a dense generator (\thex\ref{teortaustdensegenerator}).

\begin{rem}
	Reading the proofs of this section, it will be clear that it is enough to assume $\K$ to have pullbacks along split opfibrations (and comma objects for \prox\ref{propreplacement}), rather than all finite limits.
\end{rem}

\begin{prop}\label{propreplacement}
	Let $\K$ be a finitely complete $2$-category and let $\rho\:\J\to \B$ be a morphism in $\K$. Then taking comma objects along $\rho$ is equivalent to taking \pteor{strict $2$-}pullbacks along the free Grothendieck opfibration $\partial_1\:\slice{\rho}{\B}\to \B$ on $\rho$, which is split.
	\begin{eqD*}
	\begin{cd}*[6][7]
		\P \arrow[r,"{}"] \arrow[d,"{}"]\& \J \arrow[d,"{\rho}"]\arrow[ld,Rightarrow,shorten <=2.7ex,shorten >=2.2ex,"\opn{comma}"{pos=0.65}]\\
		\A \arrow[r,"{F}"']\&\B 
	\end{cd}\qquad\quad
	\begin{cd}*[6][7]
		\P \PB{rd}\arrow[d,"{F\st \partial_1}"'] \arrow[r,"{\partial_1\st F}"] \& \slice{\rho}{\B} \arrow[d,"{\partial_1}"'] \arrow[r,"\partial_0"]\& \J \arrow[d,"{\rho}"]\arrow[ld,Rightarrow,shorten <=2.7ex,shorten >=2.2ex,"\opn{comma}"{pos=0.65}] \\
		\A\arrow[r,"F"'] \& \B \arrow[r,equal]\& \B
	\end{cd}
	\end{eqD*}
\end{prop}
\begin{proof}
	It suffices to check that the diagram on the right above has the universal property of the comma object on the left, for every $F\:\A\to \B$ in $\K$. It is known that $\partial_1$ is the free Grothendieck opfibration on $\rho$ and that it is split, see Street's~\cite{street_fibandyonlemma}.
\end{proof}

\begin{prop}\label{proptaustextendstoatwofunctor}
	Let $\K$ be a $2$-category with pullbacks and fix a choice of all pullbacks. Let then $\tau\:\E\to \B$ be a split Grothendieck opfibration in $\K$. Pulling back along $\tau$ extends to an $\F$-functor
	$$\tau\st\:\laxslice{\K}{\B}\to \laxslice{\K}{\E}$$
	\pteor{considering the canonical $\F$-category structure on the lax slice described in \remx\ref{remreflectioninsideFcattheory}, i.e.\ with the tight part given by the strict slice}.
\end{prop}
\begin{proof}
	Given a morphism $F\:\A\to \B$ in $\K$, we define $\tau\st F$ as the upper morphism of the chosen pullback square in $\K$ on the left below. Given then a morphism in $\laxslice{\K}{B}$ as in the middle below, we can lift the $2$-cell in $\K$ on the right below
	\begin{eqD*}
	\sq*[p][7][8]{\P}{\E}{\A}{\B}{\tau\st F}{F\st \tau}{\tau}{F}
	\qquad\quad \tr*[4][5.5][0][0][l][-0.3][\alpha][0.85]{\A}{\A'}{\B}{\widehat{\alpha}}{F}{F'}
	\qquad\quad
	\begin{cd}*[3.5][4.4]
		\P \arrow[rr,"{\tau\st F}"]\arrow[d,"{F\st \tau}"'] \&\& \E \arrow[d,"{\tau}"] \\
		\A \arrow[rd,"{\widehat{\alpha}}"'{inner sep=0.3ex}]\arrow[rr,"{F}",""'{name=A}] \&\& \B \\
		\& \A' \arrow[from=A,Rightarrow,"{\alpha}"{pos=0.4},shorten <=0.6ex,shorten >=1ex]\arrow[ru,"{F'}"'{inner sep =0.5ex}]
	\end{cd}
	\end{eqD*}
	along the Grothendieck opfibration $\tau$, producing the chosen cartesian $2$-cell $\tau\st{\alpha}\:\tau\st F\aR{}V\:\P\to \E$ (in the cleavage) with $\tau\c V= F'\c \widehat{\alpha}\c F\st \tau$ and $\tau\ast \tau\st{\alpha}=\alpha\ast F\st \tau$. Using then the universal property of the pullback $\P'$ of $\tau$ and $F'$ we can factorize $V$ through $\tau\st F'$, obtaining a morphism $\widehat{\tau\st \alpha}\:\P\to \P'$. We define $\tau\st \alpha$ to be the upper triangle in the following commutative solid:
	\begin{cd}[2.6][9.3]
		\P \arrow[rd,"{\widehat{\tau\st \alpha}}"'{pos=0.56,inner sep=0.3ex}]\arrow[rrd,bend left=25,"{\tau\st F}",""'{name=C}]\arrow[dd,"{F\st \tau}"']\&[-4.5ex]\\
		\&\P'\arrow[from=C,Rightarrow,"{\tau\st{\alpha}}"{pos=0.35},shorten <=0.5ex,shorten >=0.9ex]\arrow[r,"{\tau\st F'}"]\& \E \arrow[dd,"{\tau}"] \\
		\A \arrow[rd,"{\widehat{\alpha}}"']\arrow[rrd,bend left=25,"{F}",""'{name=A}]\\
		\&\A' \arrow[uu,leftarrow,"{{(F')}\st \tau}"{pos=0.76},crossing over]\arrow[r,"{F'}"'{inner sep =0.5ex}] \arrow[from=A,Rightarrow,"{\alpha}"{pos=0.44},shorten <=0.5ex,shorten >=0.9ex]\& \B
	\end{cd}
	It is straightforward to check that $\tau\st$ is functorial, since $\tau$ is a split Grothendieck opfibration. For this, remember that a cleavage is the choice of a left adjoint to $\eta_\tau\:\E\to \slice{\tau}{\B}$, where the latter is the morphism induced by the identity $2$-cell on $\tau$. Such a choice then determines the liftings of the Grothendieck opfibrations $(\tau\c -)\:\HomC{\K}{\X}{\E}\to \HomC{\K}{\X}{\B}$ in $\CAT$ that we have for every $\X\in \K$, by using the universal property of $\slice{\tau}{\B}$ (to factorize the $2$-cells we want to lift). So notice that taking $(\widehat{\alpha},\alpha)\:F\to F'$ as above and $(\widehat{\beta},\beta)\:F'\to F''$ in $\laxslice{\K}{\B}$ we have that the chosen cartesian lifting of $\beta\ast (\widehat{\alpha}\c F\st \tau)=\left(\beta\ast (F')\st \tau\right)\ast \widehat{\tau\st \alpha}$ needs to coincide with $\tau\st \beta\ast \widehat{\tau\st \alpha}$.
	
	Given a $2$-cell $\delta\:(\widehat{\alpha},\alpha)\to (\widehat{\beta},\beta)\:F\to F'$ in $\laxslice{\K}{\B}$, we define $\tau\st \delta$ to be the chosen cartesian lifting of the $2$-cell $\delta\ast F\st\tau$ along the Grothendieck opfibration $(F')\st \tau$, where the latter has the cleavage induced by the cleavage of $\tau$. It is straightforward to show that the codomain of $\tau\st\delta$ is indeed $\widehat{\tau\st\beta}$ and that $\tau\st\delta$ is a $2$-cell in $\laxslice{\K}{\E}$ from $\tau\st \alpha$ to $\tau\st \beta$. It is then straightforward to check that $\tau\st$ is an $\F$-functor, using that a split Grothendieck opfibration lifts identity 2-cells to identities.
\end{proof}

\begin{rem}
	Notice that, in order to obtain a functor between lax slices, it is essential to consider pullbacks along split Grothendieck opfibrations rather than general pullbacks.
\end{rem}

\begin{teor}\label{teortausthasadjoint}
	Let $\tau\:\E\to \B$ be a split Grothendieck opfibration in $\CAT$. Then the $\F$-functor
	$$\tau\st\:\laxslice{\CAT}{\B}\to \laxslice{\CAT}{\E}$$
	has a strict right-semi-lax loose right $\F$-adjoint.
	
	As a consequence, by \thex\ref{teorsemilaxFadjpreserve}, $\tau\st$ preserves all the universal marked oplax $2$-cocylinders for an $\F$-diagram which have tight $\lambda$-components; see \corx\ref{explicitprestaust} and Section~\ref{sectionliftingpreservfordom} for what this means explicitly in practice.
\end{teor}
\begin{proof}
	We use \prox\ref{univmappingproplaxadj} (universal mapping property that characterizes a lax adjunction) to build a right-semi-lax right adjoint $\tau\stb\:\laxslice{\CAT}{\E}\to \laxslice{\CAT}{\B}$ to $\tau\st$. We will generalize the ideas of the construction of a right adjoint to the pullback between strict slices (see Conduch\'{e}'s~\cite{conduche_existadjointsadroitcondfunc} and Palmgren's~\cite{palmgren_groupoidsloccartclos}), using that $\tau$ is Conduch\'{e}.
	
	Given a morphism $f\:X\to X'$ in $\B$, we consider the following pullbacks in $\CAT$\v[-1]
	\begin{eqD*}
		\sq*[p][5.3][6.5]{\tau^{-1}(X)}{\E}{\1}{\B}{U}{}{\tau}{X}
		\quad
		\sq*[p][5.3][6.5]{\tau^{-1}(f)}{\E}{\2}{\B}{V}{}{\tau}{f}
		\qquad
		\begin{cd}*[5.3][5]
			\tau^{-1}(X)\PB{rd} \arrow[d,"{}"]\arrow[r,"{\t{0}}"] \&[-2.5ex] \tau^{-1}(f)\PB{rd} \arrow[d,"{}"]\arrow[r,"{V}"]\& \E \arrow[d,"{\tau}"]\\
			\1 \arrow[r,"{0}"'] \& \2 \arrow[r,"{f}"'] \& \B
		\end{cd}
	\end{eqD*}
	Notice that $\tau^{-1}(X)$ is the fibre of $\tau$ over $X$. Whereas $\tau^{-1}(f)$ has three kinds of morphisms, namely the morphisms in $\E$ over $\id{X}$, those over $\id{X'}$ and those over $f\:X\to X'$.
	
	Given a functor $H\:\D\to \E$, we define $\tau\stb H$ as the projection on the first component $\pr{1}\:\H\to \B$, where the category $\H$ is defined as follows:
	\begin{description}
		\item[an object] is a pair $(X,(\widehat{\alpha},\alpha))$ with $X\in \B$ and $(\widehat{\alpha},\alpha)$ a morphism in $\laxslice{\CAT}{\E}$\v[-1]
		\tr[3.6][5.1][-2][0][l][-0.3][\alpha][0.85]{\tau^{-1}(X)}{\D}{\E}{\widehat{\alpha}}{U}{H}
		\item[a morphism $(X,(\widehat{\alpha},\alpha))\to (X',(\widehat{\beta},\beta))$] is a pair $(f,(\widehat{\Phi},\Phi))$ with $f\:X\to X'$ in $\B$ and $(\widehat{\Phi},\Phi)$ a morphism in $\laxslice{\CAT}{\E}$ as on the left below such that $\Phi\ast \t{0}=\alpha$ and $\Phi\ast \t{1}=\beta$
		\begin{eqD*}
			\tr*[3.6][5.1][-2][0][l][-0.3][\Phi][0.85]{\tau^{-1}(f)}{\D}{\E}{\widehat{\Phi}}{V}{H}
			\qquad\quad
			\begin{cd}*[2.4][3.6]
				\tau^{-1}(X)\arrow[rd,"{\t{0}}"{pos=0.4},shorten <=-0.5ex,shorten >=-0.5ex]\arrow[rrrd,bend left=25,"{\widehat{\alpha}}"]\arrow[rrdd,bend right=35,"{U}"']\&[-3ex] \\[-4ex]
				\& \tau^{-1}(f) \arrow[rr,"{\widehat{\Phi}}"]\arrow[rd,"{V}"',""{name=A}]\&[-2ex]\& \D \arrow[ld,"{H}",""'{name=B}]\\
				\& \& \E
				\arrow[Rightarrow,from=A,to=B,shift left = 0.3,"{\Phi}",shorten <=0.75em,shorten >=0.75em]
			\end{cd}
		\end{eqD*}
		So the only data of $(\widehat{\Phi},\Phi)$ that are not already determined by its domain and its codomain are the assignments of $\widehat{\Phi}$ on the morphisms of $\tau^{-1}(f)$ corresponding to morphisms $g\:E\to E'$ in $\E$ over $f\:X\to X'$, and such assignments produce a morphism in $\H$ precisely when they organize into a functor $\widehat{\Phi}$ that makes the following square commute for every $g\:E\to E'$ in $\E$ over $f\:X\to X'$
		\sq[n][5.5][6.5]{E}{H(\widehat{\alpha}(E))}{E'}{H(\widehat{\beta}(E'))}{\alpha_E}{g}{H(\widehat{\Phi}(0\aar{}1,g))}{\beta_{E'}}
		\item[the identity on $(X,(\widehat{\alpha},\alpha))$] is the pair $(\id{X},(\widehat{\id{\alpha}},\id{\alpha}))$ determined by
		$$\widehat{\id{\alpha}}(0\aar{}1,g)=\widehat{\alpha}(g);$$
		\item[the composition of $(f,(\widehat{\Phi},\Phi))$ and $(f',(\widehat{\Phi'},\Phi'))$] has first component $f'\c f$ and second component determined by sending $g\:E\to E'$ over $X\ar{f}X'\ar{f'}X''$ to
		$$\widehat{\Phi'}(1\ar{}2,g_2)\c\widehat{\Phi}(0\ar{}1,g_1)$$
		where $E\ar{g_1}Z\ar{g_2}E'$ is a factorization of $g$ over $X\ar{f}X'\ar{f'}X''$ obtained by the fact that $\tau$ is a Conduch\'{e} functor. Notice that such assignment is independent from the choice of the factorization because $\widehat{\Phi}$ and $\widehat{\Phi'}$ need to agree on any morphism $(1\aeqq{} 1,h)$.
	\end{description}
	It is straightforward to check that $\H$ is a category, and $\tau\stb H$ is then surely a functor.
	
	We define the counit $\eps$ on $H$ as the morphism in $\laxslice{\CAT}{\E}$
	\tr[3.6][5.1][0][0][l][-0.3][\eps_H][0.85]{\catfont{N}}{\D}{\E}{\widehat{\eps_H}}{\tau\st\tau\stb H}{H}
	given by the evaluation, as follows. An object of $\catfont{N}$ is a pair $((X,(\widehat{\alpha},\alpha)),E)$ with $(X,(\widehat{\alpha},\alpha))\in \H$ and $E\in \tau^{-1}(X)$, whereas a morphism in $\catfont{N}$ is a pair $((f,(\widehat{\Phi},\Phi)),g)$ with $(f,(\widehat{\Phi},\Phi))$ a morphism in $\H$ and $g\:E\to E'$ in $\E$ over $f$. We define
	$$\widehat{\eps_H}((X,(\widehat{\alpha},\alpha)),E)\deq \widehat{\alpha}(E)$$
	$$\widehat{\eps_H}((f,(\widehat{\Phi},\Phi)),g)\deq \widehat{\Phi}(0\ar{} 1,g)$$
	$${\left(\eps_H\right)}_{((X,(\widehat{\alpha},\alpha)),E)}\deq \alpha_{E}$$
	It is then easy to see that $\widehat{\eps_H}$ is a functor and $\eps_H$ is a natural transformation. Notice, however, that $\eps_H$ is not tight, so that we can only hope to obtain a loose adjunction.
	
	We prove that $\eps_H$ is universal in the lax sense. So take a functor $F\:\A\to \B$ and a morphism $(\widehat{\gamma},\gamma)\:\tau\st F\to H$ in $\laxslice{\CAT}{\E}$. Wishing to obtain a right-semi-lax loose $\F$-adjunction, we search for a tight morphism in $\laxslice{\CAT}{\B}$ as on the left below that satisfies the equality of diagrams on the right
	\begin{eqD}{eqovgamma}
		\tr*[3.6][5.1][0][0][e][-0.3][\ov{\gamma}][1.05]{\A}{\H}{\B}{\widehat{\ov{\gamma}}}{F}{\tau\stb H}
		\qquad \quad\h[4]
		\begin{cd}*[6][6]
			\P \arrow[r,"{\widehat{\tau\st \ov{\gamma}}}"]\arrow[rd,bend right=20,"{\tau\st F}"',""{name=A}]\& |[alias=J]|\catfont{N}\arrow[d,"{\tau\st\tau\stb H}"{description},""{name=B}] \arrow[r,"{\widehat{\eps_H}}"]\& |[alias=K]|\D \arrow[ld,bend left=20,"{H}",""']\\
			\& \E
			\arrow[equal,from=A,to=J,shorten <=1.2ex,shorten >=1.95ex]
			\arrow[Rightarrow,from=B,to=K,shift right=0.6ex,"{\eps_H}"'{pos=0.4},shorten <=2.6ex,shorten >=2.5ex]
		\end{cd}
		\h[-3]=	\tr*[3.6][5.1][0][0][l][-0.3][\gamma][0.85]{\P}{\D}{\E}{\widehat{\gamma}}{\tau\st F}{H}
	\end{eqD}
	so that we can take $\lambda_{(\widehat{\gamma},\gamma)}=\id{}$. Given $a\:A\to A'$ in $\A$, we have $\widehat{\ov{\gamma}}(A)=(F(A),(\widehat{\alpha},\alpha))$ and $\widehat{\ov{\gamma}}(a)=(F(a),(\widehat{\Phi},\Phi))$ with
	\begin{eqD*}
		\tr*[3.6][5.1][-3][0][l][-0.3][\alpha][0.85]{\tau^{-1}(F(A))}{\D}{\E}{\widehat{\alpha}}{U}{H}\qquad \quad
		\tr*[3.6][5.1][-3][0][l][-0.3][\Phi][0.85]{\tau^{-1}(F(a))}{\D}{\E}{\widehat{\Phi}}{V}{H}
	\end{eqD*}
	And given $g\:E\to E'$ in $\E$ over $F(a)\:F(A)\ar{}F(A')$, then $(a,g)\:(A,E)\to (A',E')$ is a morphism in $\P$. So we want to define
	$$\widehat{\alpha}(E)=\widehat{\eps_H}(\widehat{\ov{\gamma}}(A),E)=\widehat{\eps_H}(\widehat{\tau\st \ov{\gamma}}\h(A,E))=\widehat{\gamma}(A,E)$$
	$$\widehat{\Phi}(0\ar{}1,g)=\widehat{\eps_H}(\widehat{\ov{\gamma}}(a),g)=\widehat{\eps_H}(\widehat{\tau\st \ov{\gamma}}\h(a,g))=\widehat{\gamma}(a,g)$$
	$${\alpha}_E={(\eps_H)}_{(\widehat{\ov{\gamma}}(A),E)}={(\eps_H)}_{(\widehat{\tau\st \ov{\gamma}}\h(A,E))}={\gamma}_{(A,E)}\v$$
	Taking a morphism $g'\:E\to E'$ in $\tau^{-1}(F(A))$,
	$$\widehat{\alpha}(g')=\widehat{\id{\alpha}}(0\ar{}1,g')=\widehat{\eps_H}(\widehat{\ov{\gamma}}(\id{A}),g')=\widehat{\eps_H}(\widehat{\tau\st \ov{\gamma}}\h(\id{A},g'))=\widehat{\gamma}(\id{A},g')$$
	It is straightforward to check that this $\widehat{\ov{\gamma}}$ works.
	
	Given another morphism $(\widehat{\xi},\xi)\:F\to \tau\stb H$ in $\laxslice{\CAT}{\B}$ and a $2$-cell $\Xi\:(\widehat{\gamma},\gamma)\aR{} (\widehat{\eps_H},\eps_H)\c (\widehat{\tau\st\xi},\tau\st\xi)$ in $\laxslice{\CAT}{\E}$, we prove that there is a unique $2$-cell $\delta\:(\widehat{\ov{\gamma}},\id{})\aR{} (\widehat{\xi},\xi)$ in $\laxslice{\CAT}{\B}$ such that
	\begin{equation}\label{eqrequestXi}
		(\widehat{\eps_H},\eps_H)\ast \tau\st\delta\c \id{}=\Xi.
	\end{equation}
	In order for $\delta$ to be a $2$-cell $(\widehat{\ov{\gamma}},\id{})\aR{} (\widehat{\xi},\xi)$, it must hold that $\tau\stb H\ast \delta=\xi$. Whereas equation~\refs{eqrequestXi} translates as $\widehat{\eps_H}\ast \tau\st \delta=\Xi$. So, for every $A\in \A$, the component $\delta_A\:\widehat{\ov{\gamma}}(A)\to \widehat{\xi}(A)$ needs to be the morphism $(\xi_A,(\widehat{\delta_A},\delta_A))$ in $\H$ given as follows. For every $g\:E\to E'$ over $\xi_A$, factorizing $g$ as the cartesian morphism $\cart{\xi_A}{E}$ in the cleavage of $\tau$ over $\xi_A$ to $E$ followed by the unique induced vertical morphism $g_{\opn{vert}}$,
	\begin{center}
		\linesep{1.9}
		\begin{tabular}{LL}
			\widehat{\delta_A}(0\ar{}1,g)=\widehat{\delta_A}(1\aeqq{}1,g_{\opn{vert}})\c\widehat{\delta_A}(0\ar{}1,\cart{\xi_A}{E})= \\
			=\widehat{\widehat{\xi}(A)}(g_{\opn{vert}})\c \widehat{\eps_H}(\delta_A,\cart{\xi_A}{E})=\widehat{\widehat{\xi}(A)}(g_{\opn{vert}})\c \widehat{\eps_H}\left((\tau\st\delta)_{A,E}\right)= \\
			= \widehat{\widehat{\xi}(A)}(g_{\opn{vert}})\c \Xi_{A,E}
		\end{tabular}
	\end{center}	
	It is straightforward to prove that this $\delta$ works.
	
	Considering $(\widehat{\gamma},\gamma)=(\widehat{\eps_H},\eps_H)$, we immediately see that $\widehat{\ov{\eps_H}}=\id{}$, because $(\widehat{\eps_H},\eps_H)$ is the evaluation. 
	
	Moreover, for every functor $F\:\A\to \B$ and morphism $(\widehat{\gamma},\gamma)\:\tau\st F\to H$ in $\laxslice{\CAT}{\E}$ we prove that 
	\begin{equation}\label{eqassumptionunivmappingprop}
		{\ov{\left((\widehat{\gamma},\gamma)\c (\widehat{\eps_H},\eps_H)\right)}}\c {\ov{\id{\tau\st F}}}={\ov{(\widehat{\gamma},\gamma)}}.
	\end{equation}
	${\ov{\id{\tau\st F}}}=(\widehat{\eta_F},\id{})$, that will be the unit $\eta_F$, is such that, for every $a\:A\to A'$ in $\A$, morphism $g'$ in $\tau^{-1}(F(A))$ and $g\:E\to E'$ in $\E$ over $F(a)\:F(A)\to F(A')$,
	$$\widehat{\widehat{\eta_F}(A)}(E)=(A,E) \qquad \quad \widehat{\widehat{\eta_F}(A)}(g')=(\id{A},g')$$
	$$\widehat{\widehat{\eta_F}(a)}(0\ar{}1, g)=(a,g)\qquad \quad {\widehat{\eta_F}(A)}_{E}= \id{}$$
	Whereas for a general $(\widehat{\psi},\psi)\:G\to H$ in $\laxslice{\CAT}{\E}$, the morphism
	$${\ov{\left((\widehat{\psi},\psi)\c (\widehat{\eps_H},\eps_H)\right)}}=(\widehat{\tau\stb{\psi}},\id{})\v[-1]$$
	will be the action of $\tau\stb$ on the morphism $(\widehat{\psi},\psi)$, and is such that $\widehat{\tau\stb{\psi}}$ acts by postcomposing the triangles with $(\widehat{\psi},\psi)$. Thus equation~\refs{eqassumptionunivmappingprop} holds.
	
	By \prox\ref{univmappingproplaxadj}, as $\lambda$ is always the identity, $\tau\stb$ extends to an oplax functor, that can be easily checked to be a $2$-functor (it acts by postcomposition), $\eps$ extends to a $2$-natural transformation, $\eta$ extends to a lax natural transformation and there exists a modification $t$ such that $\tau\stb$ is a right-semi-lax right adjoint to $\tau\st$. It is easy to check that $t$ is the identity.
	
	Since $\ov{(\widehat{\gamma},\gamma)}$ is always tight, then $\tau\stb$ is an $\F$-functor and $\eta$ has tight components. It remains to show that $\eta$ is loose strict/lax $\F$-natural. Given a morphism $(\widehat{\sigma},\sigma)\:{(\A\ar{F}\B)}\to(\A'\ar{F'}\B)$ in $\laxslice{\CAT}{\B}$, the component on $A\in \A$ of the structure $2$-cell $\eta_{(\widehat{\sigma},\sigma)}$ is the morphism in the domain of $\tau\stb\tau\st F'$ with first component $\sigma_A\:F(A)\to F'(\widehat{\sigma}(A))$ in $\B$ and second component given by
	$$\widehat{\eta_{(\widehat{\sigma},\sigma),A}}(0\ar{}1,g)=\widehat{\widehat{\eta_{F'}}\left(\widehat{\sigma}(A)\right)}(g_{\opn{vert}})=\left(\id{\h\widehat{\sigma}(A)},g_{\opn{vert}}\right).$$
	When $(\widehat{\sigma},\sigma)$ is tight, i.e.\ when $\sigma=\id{}$, we have that $g=g_{\opn{vert}}$ because $\tau$ has a normal cleavage, and we find $\eta_{(\widehat{\sigma},\sigma)}=\id{}$. Thus $\eta$ is strict/lax $\F$-natural. We conclude that $\tau\stb$ is a strict right-semi-lax loose right $\F$-adjoint to $\tau\st$. 
\end{proof}

\begin{rem}
	We have actually proved in \thex\ref{teortausthasadjoint} that $\tau\stb$ sends every morphism in $\laxslice{\CAT}{\E}$ to a tight one in $\laxslice{\CAT}{\B}$. So $\tau\stb\:\laxslice{\CAT}{\E}\to \laxslice{\CAT}{\B}$ is still an $\F$-functor if we take the trivial $\F$-category structure on the domain, i.e.\ taking everything to be tight, and the canonical one in the codomain. Of course, with such a choice of $\F$-category structures, $\tau\st$ remains an $\F$-functor and $\eta$ remains with tight components. But $\eps$ becomes tight, having now tight components trivially. 
	
	So we find a strict right-semi-lax (tight) $\F$-adjunction between $\tau\st$ and $\tau\stb$. But \thex\ref{teorsemilaxFadjpreserve} does not add anything to the preservation of $2$-colimits we have already proved in \thex\ref{teortausthasadjoint}, since it would consider strict/oplax $\F$-colimits in an $\F$-category with trivial $\F$-category structure.
\end{rem}

As for \corx\ref{remexplicitpresdom}, by \recx\ref{rectwocatofel}, cartesian-marked oplax conical 2-colimits (and thus weighted 2-colimits alike) are inscribed in the $\F$-categorical context of universal marked oplax 2-cocylinders; see the following results.

\begin{coroll}\label{explicitprestaust}
	The 2-functor $\tau\st\:\laxslice{\CAT}{\B}\to \laxslice{\CAT}{\E}$ preserves all the universal cartesian-marked oplax $2$-cocones for an $\F$-diagram which have tight components.

	That is, given 2-functors $W\:\A\op\to \CAT$ with $\A$ small and $H\:\Groth{W}\to \laxslice{\Cat}{\B}$ that sends every morphism of type $(f,\id{})$ to a triangle filled with an identity, if
	$$\zeta\:\Delta 1\aoplaxn{}\HomC{\laxslice{\Cat}{\B}}{ H(-)}{q}$$
	is a universal cartesian-marked oplax $2$-cocone for $H$ on $q\in \laxslice{\Cat}{\B}$ exhibiting $q=\oplaxncolim{H}$ such that $\zeta_{(A,X)}$ is a triangle filled with an identity for every $(A,X)\in \Groth{W}$, then $\tau\st\c \h\zeta$ is universal as well, exhibiting $$\tau\st(q)=\oplaxncolim{\left(\tau\st\c H\right)}.$$ 
\end{coroll}

\begin{exampl}\label{exacolimthatcomefromthetwocattaust}
	As we said in \remx\ref{exacolimitsthatcomefromthetwocat}, all the 2-colimits in $\laxslice{\Cat}{\B}$ that come from 2-colimits in $\Cat$ via the work of Section~\ref{sectioncolimitsin2slices} satisfy the conditions of \corx\ref{explicitprestaust}. $\tau\st\:\laxslice{\CAT}{\B}\to \laxslice{\CAT}{\E}$ thus preserves all such 2-colimits. Moreover, $\tau\st$ sends such 2-colimits to 2-colimits in $\laxslice{\CAT}{\E}$ that satisfy the conditions of \prox\ref{propsharperpreservationdom}, since $\tau\st$ is an $\F$-functor. And thus the domain 2-functor preserves the obtained 2-colimits in $\laxslice{\CAT}{\E}$.

	This means that the conditions of \corx\ref{explicitprestaust} and of \prox\ref{propsharperpreservationdom} are satisfied by most of the colimits we have in practice. For example, we have been able to apply these results to all the colimits in 2-dimensional slices that we needed in order to develop our~\cite{mesiti_twoclassifiersdensegenstacks}. As we said in \remx\ref{exacolimitsthatcomefromthetwocat}, there might be worse-behaved 2-colimits in the lax slice, but the colimits in 2-dimensional slices that we commonly have in practice are the ones that come from the base 2-category, and those ones satisfy the conditions of \corx\ref{explicitprestaust}.
\end{exampl}

	We now extend the result of preservation of $2$-colimits that we have proved for
	$$\tau\st\:\laxslice{\K}{\B}\to \laxslice{\K}{\E}$$
	when $\K=\CAT$ (\corx\ref{explicitprestaust}) to $\K=\m{\L\op}{\CAT}$ a $2$-category of $2$-dimensional presheaves. See our joint work with Caviglia~\cite{cavigliamesiti_indexedgrothconstr} for a characterization of Grothendieck opfibrations in this context. We should always keep \exax\ref{exacolimthatcomefromthetwocattaust}, \thex\ref{teorcolimitsinlaxslicesthatcomefromthebase} and \exax\ref{exaspecificcolimitsinlaxslice} in mind.

\begin{prop}\label{proptaustprestacks}
	Let $\L$ be a small $2$-category and let $\tau\:\E\to \B$ be a split Grothendieck opfibration in $\m{\L\op}{\CAT}$. Then the $\F$-functor 
	$$\tau\st\:\laxslice{\m{\L\op}{\CAT}}{\B}\to \laxslice{\m{\L\op}{\CAT}}{\E}$$
	\pteor{which is such by \prox\ref{proptaustextendstoatwofunctor}} preserves all the universal cartesian-marked oplax $2$-cocones for an $\F$-diagram which have tight components.
\end{prop}
\begin{proof}
	Consider a marking $W\:\A\op\to \CAT$ with $\A$ a small $2$-category and an $\F$-diagram $H\:\Groth{W}\to\laxslice{\m{\L\op}{\CAT}}{\B}$ (that is, a $2$-functor that sends every morphism of type $(f,\id{})$ to a triangle filled with an identity). Let then 
	$$\zeta\:\Delta 1\aoplaxn{}\HomC{\laxslice{\m{\L\op}{\CAT}}{\B}}{H(-)}{C}$$
	be a universal cartesian-marked oplax $2$-cocone that exhibits $C=\oplaxncolim{H}$ such that $\zeta_{(A,X)}$ is tight for every $(A,X)\in \Groth{W}$ (which means that it is a triangle filled with an identity). We want to prove that $\tau\st\c \zeta$ is universal as well, exhibiting $\tau\st(C)=\oplaxncolim{\left(\tau\st\c H\right)}$.
	
	Since the $\zeta_{(A,X)}$'s are all cartesian, as they are tight, and $\tau\st$ is an $\F$-functor, by \thex\ref{teordomis2colimfib} (the domain $2$-functor from a lax slice is a $2$-colim-fibration), we know that $\dom\:\laxslice{\m{\L\op}{\CAT}}{\E}\to \m{\L\op}{\CAT}$ reflects the universality of $\tau\st\c\h \zeta$. But the $2$-functors $(-)(L)\:\m{\L\op}{\CAT}\to \CAT$ of evaluation on $L\in \L$ jointly reflect $2$-colimits (as $2$-colimits in $2$-presheaves are calculated pointwise). Therefore, in order to prove that $\tau\st\c\h \zeta$ is universal, it suffices to show that, for every $L\in \L$, the cartesian-marked oplax $2$-cocone $(-)(L)\c \dom\c\h \tau\st \c \zeta$ is universal. Notice now that the diagram of $2$-functors
	\begin{cd}[5][5]
		\laxslice{\m{\L\op}{\CAT}}{\B}\arrow[r,"{\tau\st}"]\arrow[d,"{(-)_L}"]\& \laxslice{\m{\L\op}{\CAT}}{\E} \arrow[r,"{\dom}"]\arrow[d,"{(-)_L}"]\& \m{\L\op}{\CAT} \arrow[d,"{(-)(L)}"] \\
		\laxslice{\CAT}{\B(L)}\arrow[r,"{{(\tau_L)}\st}"']\& \laxslice{\CAT}{\E(L)}\arrow[r,"{\dom}"']\& \CAT
	\end{cd}
	is commutative, where $(-)_L$ is the $\F$-functor that takes components on $L$, because pullbacks in $\m{\L\op}{\CAT}$ are calculated pointwise and the components of the liftings along $\tau$ are the liftings of the components of $\tau$. Indeed every component $\tau_L$ of $\tau$ is a split Grothendieck opfibration in $\CAT$ because $\tau\c -\:\HomC{\m{\L\op}{\CAT}}{\y{L}}{\E}\to \HomC{\m{\L\op}{\CAT}}{\y{L}}{\B}$ is so, taking on the former the cleavage induced by the latter. And since a cleavage for $\tau$ is the choice of a left adjoint to $\eta_\tau\:\E\to \slice{\tau}{\B}$ (where the latter is the morphism induced by the identity $2$-cell on $\tau$), the cleavages determined on the Grothendieck opfibrations $(\tau\c -)\:\HomC{\m{\L\op}{\CAT}}{\X}{\E}\to \HomC{\m{\L\op}{\CAT}}{\X}{\B}$ in $\CAT$ (by applying the universal property of $\slice{\tau}{\B}$) are compatible.
	
	We prove that $\dom\c \h{(\tau_L)}\st\c (-)_L\c \zeta$ is universal. We have that $(-)_L\c \zeta$ is universal because it suffices to check that $\dom\c (-)_L\c \zeta=(-)(L)\c \dom\c \h\zeta$ is so, by \thex\ref{teordomis2colimfib}, as $\zeta$ has tight components and $(-)_L$ is an $\F$-functor. And $\dom$ preserves the universality of $\zeta$ by \thex\ref{teordomhassemilaxFadjoint} (thanks to the hypotheses), while $(-)(L)$ preserves every $2$-colimit. Then $\dom\c \h{(\tau_L)}\st\c (-)_L\c \zeta$ is universal applying \thex\ref{teortausthasadjoint} and \thex\ref{teordomhassemilaxFadjoint}, thanks to the hypothesis and to the fact that both $(-)_L$ and ${(\tau_L)}\st$ are $\F$-functors.
\end{proof}

	We conclude extending again the result of preservation of $2$-colimits for $\tau\st\:\laxslice{\K}{\B}\to \laxslice{\K}{\E}$, to $\K$ any finitely complete $2$-category with a dense generator. For this, we will need to restrict to relatively absolute $2$-colimits. But, as we recall in \defx\ref{defdense}, any object of $\K$ can be expressed as a relatively absolute $2$-colimit of dense generators, so that our assumption is not much restrictive in practice. Again, we should keep \exax\ref{exacolimthatcomefromthetwocattaust}, \thex\ref{teorcolimitsinlaxslicesthatcomefromthebase} and \exax\ref{exaspecificcolimitsinlaxslice} in mind.

	\begin{defne}[Kelly's~{\cite[Chapter~5]{kelly_basicconceptsofenriched}}]\label{defdense}
	A (2-)fully faithful 2-functor $J\:\L\to \K$ with $\L$ small is \dfn{a dense generator} (or also just \dfn{dense}) if the restricted Yoneda embedding
	\begin{fun}
		\t{J} & \: & \K & \too & \m{\L\op}{\CAT} \\[0.7ex]
		&& F &\mto& \HomC{\K}{J(-)}{F}
	\end{fun}
	is (2-)fully faithful.

	Equivalently, if every $F\in \K$ is a \dfn{$J$-absolute}, i.e.\ preserved by $\t{J}$, weighted 2-colimit in $\K$ of a diagram which factors through $\L$.
\end{defne}

\begin{exampl}\label{exarepresinprestacks}
	Let $\L$ be a small 2-category. The Yoneda embedding $\yy\:\L\to \m{\L\op}{\Cat}$ is dense; $\yy$-absoluteness is automatic as $\t{\yy}$ is essentially the identity. That is, representables form a dense generator of the 2-category of 2-presheaves.	In particular, the singleton category $\1$ is a dense generator of $\Cat$. \thex\ref{teortaustdensegenerator} will thus be a generalization of \prox\ref{proptaustprestacks} and \corx\ref{explicitprestaust}.
\end{exampl}

\begin{teor}\label{teortaustdensegenerator}
	Let $\K$ be a finitely complete $2$-category and let $J\:\L\to \K$ be a fully faithful dense $2$-functor. Consider $\tau\:\E\to \B$ a split Grothendieck opfibration in $\K$. Then the $\F$-functor 
	$$\tau\st\:\laxslice{\K}{\B}\to \laxslice{\K}{\E}$$
	\pteor{which is such by \prox\ref{proptaustextendstoatwofunctor}}
	 preserves all the universal cartesian-marked oplax $2$-cocones for an $\F$-diagram which have tight components and whose domain is $J$-absolute.
\end{teor}
\begin{proof}
	Let $\A$ be a small $2$-category and consider a marking $W\:\A\op\to \CAT$ and an $\F$-diagram $H\:\Groth{W}\to\laxslice{\K}{\B}$. Let then 
	$$\zeta\:\Delta 1\aoplaxn{}\HomC{\laxslice{\K}{\B}}{H(-)}{C}$$
	be a universal cartesian-marked oplax $2$-cocone that exhibits $C=\oplaxncolim{H}$ such that $\zeta_{(A,X)}$ is tight for every $(A,X)\in \Groth{W}$. Assume also that $\dom\c \h\zeta$ is $J$-absolute, i.e.\ preserved by $\t{J}\:\K\to \m{\L\op}{\CAT}$. Notice that $\dom\c\h \zeta$ is indeed universal by \corx\ref{remexplicitpresdom}. We want to prove that $\tau\st\c\h \zeta$ is universal as well, exhibiting
$$\tau\st(C)=\oplaxncolim{\left(\tau\st\c H\right)}.$$
	
	Since the $\zeta_{(A,X)}$'s are all cartesian (as they are tight) and $\tau\st$ is an $\F$-functor, by \thex\ref{teordomis2colimfib}, we know that $\dom\:\laxslice{\K}{\E}\to \K$ reflects the universality of $\tau\st\c \zeta$. Moreover, by definition of dense functor, $\t{J}$ is fully faithful and hence reflects any $2$-colimit; and the $2$-functors $(-)(L)\:\m{\L\op}{\CAT}\to \CAT$ of evaluation on $L\in \L$ jointly reflect $2$-colimits. Therefore, in order to prove that $\tau\st\c \zeta$ is universal, it suffices to show that, for every $L\in \L$, the cartesian-marked oplax $2$-cocone $(-)(L)\c \t{J}\c \dom\c \h\tau\st\c \zeta$ is universal.
	
	Notice now that the diagram of $2$-functors
	\begin{cd}[4][5]
		\laxslice{\K}{\B}\arrow[r,"{\tau\st}"]\arrow[d,"{\laxslice{\t{J}}{}}"]\& \laxslice{\K}{\E} \arrow[r,"{\dom}"]\arrow[d,"{\laxslice{\t{J}}{}}"]\& \K \arrow[d,"{\t{J}}"] \\
		\laxslice{\m{\L\op}{\CAT}}{\t{J}(\B)}\arrow[d,"{(-)_L}"]\& \laxslice{\m{\L\op}{\CAT}}{\t{J}(\E)} \arrow[r,"{\dom}"]\arrow[d,"{(-)_L}"]\& \m{\L\op}{\CAT} \arrow[d,"{(-)(L)}"] \\
		\laxslice{\CAT}{\t{J}(\B)(L)}\arrow[r,"{{(\tau\c -)}\st}"']\& \laxslice{\CAT}{\t{J}(\E)(L)}\arrow[r,"{\dom}"']\& \CAT
	\end{cd}
	is commutative, where $\laxslice{\t{J}}{}$ is the $\F$-functor that applies $\t{J}$ on morphisms and triangles. Indeed $\HomC{\K}{J(L)}{-}$ preserves pullbacks, and since a cleavage for $\tau$ is the choice of a left adjoint to $\eta_\tau\:\E\to \slice{\tau}{\B}$, the cleavages determined on the Grothendieck opfibrations $(\tau\c -)\:\HomC{\m{\L\op}{\CAT}}{\X}{\E}\to \HomC{\m{\L\op}{\CAT}}{\X}{\B}$ in $\CAT$ (by applying the universal property of $\slice{\tau}{\B}$) are compatible.
	
	We prove that $\dom\c \h{(\tau\c -)}\st\c (-)_L\c \laxslice{\t{J}}{}\h[-1]\c \zeta$ is universal. We have that $\laxslice{\t{J}}{}\c \zeta$ is universal, since it suffices to check that $\dom\c\h \laxslice{\t{J}}{}\c \zeta$ is so, by \thex\ref{teordomis2colimfib}, as $\zeta$ has tight components and $\laxslice{\t{J}}{}$ is an $\F$-functor. And $\dom\c\h \laxslice{\t{J}}{}\h[-1]\c \zeta=\t{J}\c \dom\c\h \zeta$ is universal because $\dom\c \h\zeta$ is $J$-absolute by hypothesis. Then $(-)_L$ preserves the universality of $\laxslice{\t{J}}{}\c \zeta$ because $\dom\c (-)_L=(-)(L)\c \dom$ does so. Finally, we obtain that $\dom\c \h{(\tau\c -)}\st\c (-)_L\c \laxslice{\t{J}}{}\h[-1]\c \zeta$ is universal applying \thex\ref{teortausthasadjoint} and \thex\ref{teordomhassemilaxFadjoint}, thanks to the hypothesis and to the fact that $\laxslice{\t{J}}{}$, $(-)_L$ and ${(\tau\c -)}\st$ are all $\F$-functors.
\end{proof}

\subsection*{Acknowledgements}
	I would like to thank my supervisor Nicola Gambino for the interesting discussions and his helpful advice on the subject of this paper. I thank Charles Walker for suggesting the possible use of $\F$-categorical techniques to justify my work. Finally, I thank the anonymous referee for their useful comments and suggestions. Part of this research has been conducted while visiting the University of Manchester.

\bibliographystyle{abbrv}
\bibliography{Bibliography2}

\end{document}